\documentclass{article}
\usepackage[utf8]{inputenc}

\usepackage{standalone}
\usepackage{amsmath}
\usepackage{amssymb}
\usepackage{amsfonts}
\usepackage{amsthm}
\usepackage{tikz}
\usepackage{multicol}
\usepackage{etoolbox}
\usepackage{xcolor}
\usetikzlibrary{calc}
\usetikzlibrary{decorations.pathreplacing}
\usepackage{calc}
\usepackage{comment}
\usepackage{theoremref}
\usepackage{amsrefs}
\usepackage{enumitem}
\usepackage{marginnote}
\usepackage{comment}

\newcommand{\lozh}[3]{
    \draw[-] (#1+1, -#2, #3) -- (#1+1, -#2-1, #3) -- ((#1+1, -#2, #3+1) -- (#1-1+1, -#2, #3) -- (#1+1, -#2, #3);
}

\newcommand{\lozr}[3]{
    \draw[-](#1,-#2,#3) -- (#1,-#2-1,#3) -- (#1,-#2-1,#3+1) -- (#1,-#2,#3+1) -- (#1,-#2,#3);
}

\newcommand{\lozl}[3]{
    \draw[-](#1+1,-#2,#3) -- (#1-1+1,-#2,#3) -- (#1-1+1,-#2,#3+1)  -- (#1+1,-#2,#3+1) -- (#1+1,-#2,#3);
}

\newcommand{\hex}[4]{
	\pgfmathtruncatemacro{\a}{#1}
	\pgfmathtruncatemacro{\b}{#2}
	\pgfmathtruncatemacro{\c}{#3}
	\pgfmathtruncatemacro{\t}{#4}

	\pgfmathtruncatemacro{\length}{\a+\c+\t}
	\pgfmathtruncatemacro{\height}{\a+\b+\t}
	\pgfmathtruncatemacro{\longdiag}{\a+\b+\c+\t}
    
    \draw[clip] (0,-\a,0) -- ++(0,-\b-\t,0) -- ++(0,0,\c) -- ++(-\a-\t,0,0) -- ++(0,\b,0) -- ++(0,0,-\c-\t) -- cycle;

    \foreach \i in {0,...,\length}{
        \draw[-] (0,0,\i) -- ++(0,-\height,0);
    }
    
    \foreach \j in {0,...,\height}{
        \draw[-] (0,-\j,0) -- ++(0,0,\length);
    }
    
    \foreach \k in {0,...,\longdiag}{
        \draw[-] (0,0,\k) -- ++(\k,0,0);
    }
}

\newcommand{\hexpic}[5]{
\begin{tikzpicture}[x={(0:1cm)},y={(120:1cm)},z={(240:1cm)},scale={#5}]
    \hex{#1}{#2}{#3}{#4}
\end{tikzpicture}
}

\newcommand{\dentB}[2]{
\draw[fill=black] (0,-#1-#2,0) -- ++(-1,0,0) -- ++(0,0,-1) -- cycle;}

\newcommand{\dentC}[2]{
\draw[fill=black] (0,0,#1+#2) -- ++(1,0,0) -- ++(0,1,0) -- cycle;}

\newcommand{\hexTwoDent}[7]{
	\pgfmathtruncatemacro{\a}{#1}
	\pgfmathtruncatemacro{\b}{#2}
	\pgfmathtruncatemacro{\c}{#3}
	\pgfmathtruncatemacro{\sOne}{#4}
	\pgfmathtruncatemacro{\tOne}{#5}
	\pgfmathtruncatemacro{\sTwo}{#6}
	\pgfmathtruncatemacro{\tTwo}{#7}
	
	\pgfmathtruncatemacro{\t}{\tOne+\tTwo}
	\pgfmathtruncatemacro{\uOne}{\b+\tTwo-\sOne}
	\pgfmathtruncatemacro{\uTwo}{\c+\tOne-\sTwo}

	\pgfmathtruncatemacro{\length}{\a+\c+\t}
	\pgfmathtruncatemacro{\height}{\a+\b+\t}
	\pgfmathtruncatemacro{\longdiag}{\a+\b+\c+\t}
    
    \draw[clip] (0,-\a,0) -- ++(0,-\sOne,0) -- ++(0,0,\tOne) -- ++(\tOne,0,0) -- ++(0,-\uOne,0) -- ++(0,0,\c) -- ++(-\a-\t,0,0) -- ++(0,\b,0) -- ++(0,0,-\uTwo) -- ++(\tTwo,0,0) --++(0,\tTwo,0) -- ++(0,0,-\sTwo) -- cycle;

    \foreach \i in {0,...,\length}{
        \draw[-] (0,0,\i) -- ++(0,-\height,0);
    }
    
    \foreach \j in {0,...,\height}{
        \draw[-] (0,-\j,0) -- ++(0,0,\length);
    }
    
    \foreach \k in {0,...,\longdiag}{
        \draw[-] (0,0,\k) -- ++(\k,0,0);
    }
}

\newcounter{fact}
\newtheorem{theorem}[fact]{Theorem}
\newtheorem{proposition}[fact]{Proposition}
\newtheorem{lemma}[fact]{Lemma}
\newtheorem{corollary}[fact]{Corollary}

\newcommand{\Pl}{P}
\DeclareMathOperator{\M}{M}

\usepackage{accents}
\newcommand\thickbar[1]{\accentset{\rule{.4em}{.8pt}}{#1}}

\newcommand{\brac}[1]{#1}

\newcommand{\vdent}[3]{\draw[fill=white](#1,#2,#3) --++(0,0,1) --++(0,-1,0) --++(0,0,-1)-- cycle; }

\newcommand\numberthis{
\addtocounter{equation}{1}\tag{\theequation}
}

\newcommand{\xequiv}[1]{\ \rlap{$\equiv$}{\raisebox{.55em}{\tiny \ $#1$}}\ }

\newcommand{\aequiv}{\xequiv{a}}
\usepackage{titling}

\title{Lozenge Tiling Function Ratios for Hexagons with Dents on Two Sides}
\author{Daniel Condon }
\date{\today}

\begin{document}

\maketitle
\thispagestyle{empty}

\begin{abstract}
    We give a formula for the number of lozenge tilings of a hexagon on the triangular lattice with unit triangles removed from arbitrary positions along two non-adjacent, non-opposite sides. Our formula implies that for certain families of such regions, the ratios of their numbers of tilings are given by simple product formulas.
\end{abstract}

\section{Introduction}

The \emph{triangular lattice} is a tiling of the plane by unit equilateral triangles, and a \emph{region} on the triangular lattice is a connected union of finitely many of those triangles\footnote{When we say `triangle' in this paper, we always mean a unit triangle on the lattice.}. We say triangles that share an edge are \emph{adjacent}.
A \emph{lozenge} on the triangular lattice is the union of two adjacent triangles, and a \emph{lozenge tiling} or simply \emph{tiling} of a region is a set of lozenges within that region, which cover the region, and which do not overlap. A region with at least one tiling is called \emph{tileable}.
A hexagonal region is said to be \emph{semiregular} if its opposite sides are the same length. We identify congruent regions and let $H_{a,b,c}$ denote the semiregular hexagon(al region) with opposite sides of lengths $a$, $b$, and $c$.

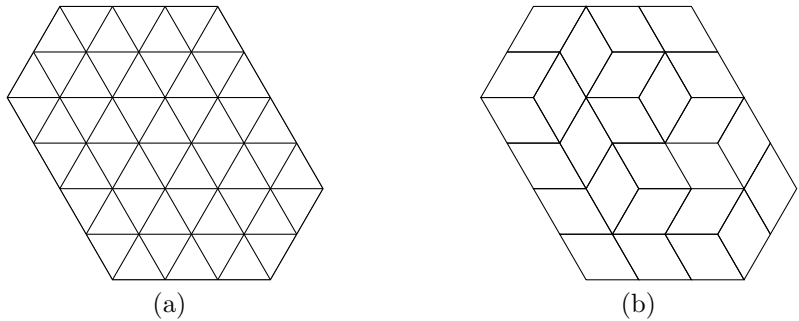
\begin{figure}
\begin{center}
\begin{multicols}{2}
\hexpic3420{.7}

(a)

\vfill \null \columnbreak

\begin{tikzpicture}[x={(300:1cm)},y={(180:1cm)},z={(60:1cm)},scale={.7}]

    \lozr 000
    \lozr 001
    \lozl 011
    \lozh 012
    \lozh 022
    \lozl 010
    \lozh 000 
    \lozh 100
    \lozh 200
    \lozh 300
    \lozh 310
    \lozh 320
    \lozl 330
    \lozl 331
    \lozr 111
    \lozh 122
    \lozl 030
    \lozh 020
    \lozl 110
    \lozr 221
    \lozl 210
    \lozh 120
    \lozr 310
    \lozr 320
    \lozl 140
\end{tikzpicture}

(b)
\end{multicols}
\end{center}
\caption{(a) $H_{3,4,2}$; (b)  a tiling of $H_{3,4,2}$} 
    \label{fig:fig1}
\end{figure}

A formula for the number of lozenge tilings of a semiregular hexagon was first given by MacMahon \cite{Ma} in the context of plane partitions. A bijection between plane partitions and lozenge tilings was later given by David and Tomei \cite{DT89}. Letting $\M(R)$ denote the number of lozenge tilings of a region $R$, MacMahon's formula can be expressed
\begin{equation}\M(H_{a,b,c}) = \prod_{i=1}^c \frac{(a+i)_b}{(i)_b}=: \Pl(a,b,c)\label{eq:MacMahon}\end{equation}
where $(x)_y$ is the pochhammer symbol, $(x)_y = \prod_{i=0}^{y-1}(x+i)$. 
When $R_{\thickbar x}$ denotes a family of regions defined by parameters $\thickbar x$ (as with $\{H_{a,b,c}: a,b,c \in \mathbb N\}$) we say that $\M(R_{\thickbar x})$ denotes the \emph{tiling function} of $R_{\thickbar x}$ with parameters $\thickbar x$. So $\Pl(a,b,c)$ is the tiling function for semiregular hexagons.

Hexagons in the triangular lattice need not be semiregular, but the lengths of opposite sides necessarily differ by the same amount, say $t$. When $t>0$ we say the longer of each opposite pair is a \emph{long} side and the rest are \emph{short} sides. We can therefore define hexagonal regions by the lengths of their short sides $a,b, c$, and the difference $t$, using the notation $H_{a,b,c,t}$. The sides of the hexagon alternate around the perimeter between short and long, so without a loss of generality we will assume that the side lengths of the hexagon appear in the clockwise order $a, b+t, c, a+t, b, c+t$ starting from the north side.  The hexagon in Figure \ref{fig:fig2} (a) is $H_{4,3,2,4}$. Semiregular hexagons have $t=0$, so $H_{a,b,c} = H_{a,b,c,0}.$

The region $H_{4,3,2,4}$ has no lozenge tilings at all. A lozenge covers exactly one up-pointing triangle and one down-pointing triangle, so that a region is tileable if it contains the same number of triangles with each orientation: we call such a region \emph{balanced}. The hexagon $H_{a,b,c,t}$ contains an excess of $t$ triangles with one orientation, so semiregular hexagons are the only hexagons which are balanced. The orientation in excess is that along the {long} sides of the hexagon (by our conventions, these are up-pointing triangles).


In this paper we study regions we call \emph{dented hexagons}, obtained by removing unit triangles along\footnote{\label{note1}When we say a unit triangle is \emph{along} a side of the the hexagon, we mean it shares an edge with the border of that side. When we say we a triangle is removed \emph{from} a side, we mean that the removed triangle is along that side.} two long sides of a hexagon.
We suppose $m$ triangles are removed from the northeast side, and $n$ triangles are removed from the northwest side. For balanced regions, this implies $t=n+m$.  Figure \ref{fig:fig2} (b) gives one example of such a region. The black triangles have been removed from the region; we call these \emph{dents}. The grey areas are part of the dented hexagon but are covered by \emph{forced lozenges}, which are described at the end of Section \ref{background}.

\begin{figure}
\begin{center}
\begin{multicols}{2}

\begin{tikzpicture}[x={(0:1cm)},y={(120:1cm)},z={(240:1cm)},scale=.5]

\draw [decorate,decoration={brace,amplitude=5pt,raise=4pt}]
(0,0,4) -- ++(4,0,0) node [black,midway,yshift=.5cm] {\footnotesize
$a$};

\draw [decorate,decoration={brace,amplitude=5pt,raise=4pt}]
(0,0,9) -- ++(0,0,-5) node [black,midway,yshift=.3cm,xshift=-.6cm] {\footnotesize
$c+t$};

\draw [decorate,decoration={brace,amplitude=5pt,raise=4pt}]
(0,-4,0) -- ++(0,-6,0) node [black,midway,yshift=.3cm,xshift=.6cm] {\footnotesize
$b+t$};

\draw [decorate,decoration={brace,amplitude=5pt,raise=4pt}]
(0,-11,0) -- ++(0,0,2) node [black,midway,yshift=-.3cm,xshift=.4cm] {\footnotesize
$c$};

\draw [decorate,decoration={brace,amplitude=5pt,raise=4pt}]
(0,-3,10) -- ++(0,3,0) node [black,midway,yshift=-.3cm,xshift=-.4cm] {\footnotesize
$b$};

\draw [decorate,decoration={brace,amplitude=5pt,raise=4pt}]
(7,-4,9) -- ++(-8,0,0) node [black,midway,yshift=-.5cm] {\footnotesize
$a+t$};

    \hex{4}{3}{2}{4}
\end{tikzpicture}

(a)

\columnbreak

\vspace*{\fill}

\begin{tikzpicture}[x={(0:1cm)},y={(120:1cm)},z={(240:1cm)},scale=.5]
    
    \draw [decorate,decoration={brace,amplitude=5pt,raise=6pt}]
    (0,-4,0) -- ++(0,-1,0) node [black,midway,yshift=.3cm,xshift=.6cm] {\footnotesize
    $u_1$};
    
    \draw [decorate,decoration={brace,amplitude=5pt,raise=2pt}]
    (0,-4,0) -- ++(0,-4,0) node [black,midway,yshift=.3cm,xshift=.4cm] {\footnotesize
    $u_2$};
    
    \draw [decorate,decoration={brace,amplitude=5pt,raise=6pt}]
    (0,0,7) -- ++(0,0,-3) node [black,midway,yshift=.3cm,xshift=-.5cm] {\footnotesize
    $v_1$};
    
    \draw [decorate,decoration={brace,amplitude=5pt,raise=2pt}]
    (0,0,8) -- ++(0,0,-4) node [black,midway,xshift=-.6cm] {\footnotesize
    $v_2$};
    
    \hex{4}{3}{2}{4}
    \dentB{4}{1}
    \dentB{4}{4}
    \dentC{4}{3}
    \dentC{4}{4}

    \draw[fill=black!20] (0,-4,0) -- ++(-4,0,0) -- ++(0,0,1) -- ++(4,0,0) -- cycle;
    
    \draw (0,-3,1) -- ++(0,0,1);
    \draw (-1,-3,1) -- ++(0,0,1);
    \draw (-2,-3,1) -- ++(0,0,1);
    
    \draw[fill=black!20] (0,0,7) --++(1,0,0) --++(0,-1,0) --++(-1,0,0) -- cycle;

\end{tikzpicture}

\vspace{.75cm}

(b)

\columnbreak
\end{multicols}
\end{center}
\caption{(a) $H_{4,3,2,4}$; (b) $H_{4,3,2,4,(1,4),(3,4)}$ with forced lozenges shaded grey. For this region $\underline{u_1}=5,\underline{u_2}=3$, $\underline{v_1}=\underline{v_2}=2$}
    \label{fig:fig2}
\end{figure}
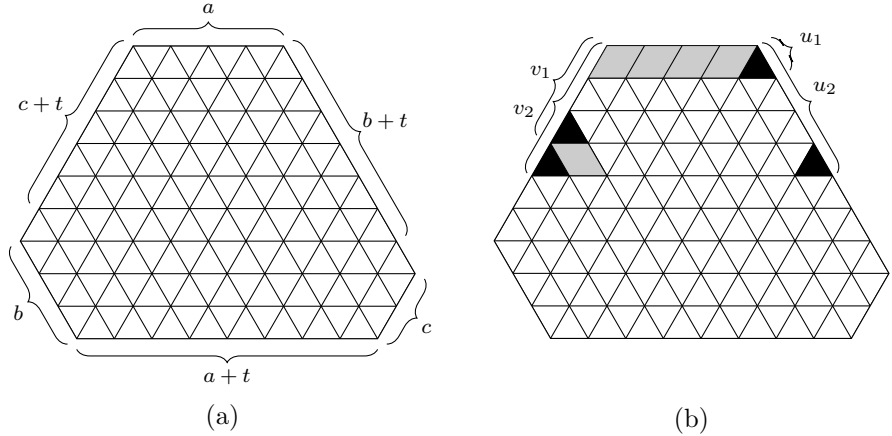

Given non-negative integers $a,b,c,t$ suppose $\vec u = (u_i)_{i=1}^m$  and $\vec v = (v_j)_{j=1}^n$ are vectors of integers with $1 \leq u_i < u_{i+1} \leq b+t$ for $1 \leq i < m$, and $1 \leq v_j < v_{j+1} \leq c+t$ for $1 \leq j < n$, so that $a > 0$ or $u_1 > 1$ or $v_1 > 1$. We use $H_{a,b,c,t,\vec u, \vec v}$ to denote the dented hexagon formed by removing dents from $H_{a,b,c,t}$ at locations indexed by $\vec u$ and $\vec v$. Specifically, we remove each $u_i$th unit triangle from the northeast side of $H_{a,b,c,t}$ and each $v_j$th unit triangle from the northwest side of $H_{a,b,c,t}$. The indexing in both cases starts at the northmost triangle along each side. It is straightforward to check that $H_{a,b,c,t,\vec u, \vec v}$ with parameters as described is well-defined and unique.


Other natural parameters of dented hexagons we will use are $\underline{u_i}:=b+n+i-u_i$ and $\underline{v_j}:=c+m+j-v_j$. The parameter $\underline{u_i}$ counts the number of triangles 
along the northeast side which are below the $i$th dent along that side but have \emph{not} been removed from the region: these are the up-pointing triangles immediately southeast of that dent, as shown in Figure \ref{fig:fig2} (b). The parameter $\underline{v_j}$ analogously counts up-pointing triangles directly southwest of the dent indexed by $v_j$.

The main result of this paper is a remarkably simple product formula for the ratio between the tiling functions of $H_{a,b,c,t,\vec u, \vec v}$ and $H_{0,b,c,t,\vec u, \vec v}$ so long as the latter region is well defined. It follows that the first family of regions has a `nice' tiling function when the latter family does, and we identify some instances of this. 
This is not the first paper to identify families of regions with nice ratios of tiling functions; similar results were recently found by Lai, Ciucu, Rohatgi, and Byun \cite{By19}, \cite{Ci19a}, \cite{Ci19b}, \cite{La19a}, \cite{La19b}, \cite{La19c}.
One specific subfamily of dented hexagons was studied by Lai, who found a lovely tiling function \cite{La17}. A general formula for the tiling function of a hexagonal region with dents in arbitrary positions along its border has been discovered by Ciucu and Fischer \cite{CF16}, but that formula is given only as the determinant of a matrix.

\section{Further Background} \label{background}

A region on the triangular lattice can be identified with a graph, so that each unit triangle corresponds to a vertex and adjacent triangles correspond to adjacent vertices. This graph is bipartite, with its vertex bipartition defined by the orientations of the triangles, since triangles are only adjacent to triangles of the opposite orientation. A lozenge tiling of a region naturally partitions the triangles of that region into adjacent pairs, and so a lozenge tiling can be identified with a perfect matching on the graph of that region. We will therefore borrow two graph theoretic techniques.

The following is a direct application of the Graph Splitting Lemma, which appears as Lemma 3.6 in a 2014 paper by Lai \cite{La14}, and is implicit in earlier work by Ciucu \cite{Ci97}.

\begin{lemma}[Region Splitting Lemma]\thlabel{splitting}
    Let $R$ be a balanced region of the triangular lattice with a partition into regions $P$ and $Q$ such that unit triangles in $P$ which are adjacent to unit triangles in $Q$ are all of the same orientation, and that this orientation is not in excess within $P$.
Then $\M(R) = \M(P)\cdot \M(Q)$, and in particular $\M(R)= 0$ if either $P$ or $Q$ is not balanced.
\qed
\end{lemma}

We will also use Kuo's graphical condensation method. The following is a special case of his Theorem 5.4 from \cite{Ku04}, expressed in the language of this paper. 
\begin{lemma}\thlabel{kuo} Let $R$ be a simply connected region of the triangular lattice, and let $\alpha, \beta, \gamma, \delta$ be unit triangles $R$ of the same orientation which touch the boundary of $R$ at a corner or edge, so that the places they touch the boundary appear in the cyclic order $\alpha, \beta, \gamma, \delta$. Then
\begin{align*}
    & \M(R - \alpha-\beta)\cdot \M(R - \gamma-\delta) \\
   =& \M(R -\alpha-\delta) \cdot \M(R - \beta-\gamma) - \M(R-\alpha-\gamma) \cdot \M(R-\beta-\delta) \numberthis \label{eqn}
\end{align*}
\qed
\end{lemma}

The lozenge tilings of a region on the triangular lattice often have a natural interpretation as families of nonintersecting lattice paths with fixed start and end points; this relationship is shown in Figure \ref{fig:lattice}. We will make use of the following classical result originally due to Lindstr{\"o}m \cite{Li73}, and independently to  Gessel and Viennot \cite{Ge89}.
The formulation we use could be considered a restatement of Lemma 14 from \cite{Ci01}.
\begin{proposition}\thlabel{GV}
    Let $S=\{(s_i,t_i) : i \in [n]\}$ and $E=\{(p_j,q_j) : j \in [n]\}$ be sets of coordinates on the square lattice $\mathbb Z^2$, such that every Up-Right lattice path from $(s_i,t_i)$ to $(p_j,q_j)$ intersects any Up-Right lattice path from $(s_j,t_j)$ to $(p_i,q_i)$ for $i \neq j$. The number of families of $n$ non-intersecting Up-Right lattice paths which each start at a point in $S$ and end at a point in $E$ is given by $\det((a_{i,j})_{i,j=1}^n)$, where $a_{i,j} = {p_j+q_j-s_i-t_i \choose q_j-t_i}$ is the number of Up-Right lattice paths from $(s_i,t_i)$ to $(p_j,q_j)$.
\end{proposition}

When a lozenge is in all possible tilings of a region, we say that lozenge is \emph{forced}.
Often, a region contains a unit triangle which is adjacent to only one other triangle - the lozenge covering those two triangles is forced. One can remove the triangles covered by a forced lozenge to produce a new region, the tilings of which are in a natural bijection with the tilings of the original region; the bijection is defined by removing or replacing the forced lozenge. The new region may itself have a triangle only adjacent to one other triangle, which means it too is covered by a lozenge that is forced in the new region, and therefore in the original region. Removing triangles covered by forced lozenges may thus yield a sequence of regions that all have the same number of tilings. 
Figure \ref{fig:fig2} (b) serves as an example. The forced lozenges are shown and colored grey.

\section{Main Results}

Our first remark about dented hexagons is a characterization of when they have any tilings at all.

\begin{proposition}\thlabel{existence}
    Let a balanced dented hexagon $H$ be given. Let $L_N$ be the Nth horizontal lattice line south of H's northern side\footnote{$L_N$ is then $\frac{\sqrt3}{2}N$ units south of H's northern side.}, and $\mu_N$ be the number of dents north of $L_N$. Then $H$ has a tiling iff $\mu_N \leq N$ for all $N \in \mathbb N$.
\end{proposition}

\begin{proof}
    Let $\vec u$ and $\vec v$ and a balanced dented hexagon $H=H_{a,b,c,t,\vec u, \vec v}$ be given. 
    
    Suppose for some $N$, $\mu_N > N$; assume that $N$ is minimal such that this is the case. It is straightforward to check the region above $L_N$ is unbalanced. By \thref{splitting}, applied to the regions above and below $L_N$, the dented hexagon has no tilings.
    
    If $\mu_N \leq N$ for each $N$, we can exhibit a partition of the overall region into tileable regions, proving the existence of a tiling for the entire region. We induct on the number of dents.
    
    If the region has no dents, it is a balanced semiregular hexagon and has at least one tiling. Suppose \thref{existence} holds for dented hexagons with fewer than $t$ dents. Suppose WLOG that $u_m \geq v_n$; in other words the southmost dent is on the eastern side of the region. We sketch the semiregular hexagon $H_{1,\underline{u_m},c}$ directly underneath the southmost dent, so that its northern side is the border of that dent. Figure \ref{fig:fig7} (a) depicts this hexagon (dark). 
    We also sketch a parallelogram consisting of all triangles directly west of the southwest edge of the dark hexagon. 
    This parallelogram is also depicted (light) in Figure \ref{fig:fig7} (a). Both the hexagon and parallelogram are tileable, and the rest of the region is tileable by the induction hypothesis; so the entire region is tileable.
    
\begin{figure}
\begin{center}
\begin{tikzpicture}[x={(0:1cm)},y={(120:1cm)},z={(240:1cm)},scale=.45]

    \begin{scope}
    
    \draw[line width=2pt, fill=blue!60] (0,-11,0) -- ++(-1,0,0) -- ++(0,0,3) -- ++(0,-3,0) -- ++(1,0,0) -- ++(0,0,-3) -- cycle;
    
    
    \draw[line width=2pt, fill=blue!10] (-1,-11,3) -- ++(-8,0,0) -- ++(0,-3,0) -- ++(8,0,0) -- cycle;
    
    \hex{5}{5}{3}{4}
    \dentB{5}{5}
    \dentB{5}{6}
    \dentC{5}{2}
    \dentC{5}{5}
    
    
    \end{scope}

    \begin{scope}[xshift=14cm]

    \draw[line width=2pt, fill=blue!60] (0,-11,0) -- ++(-1,0,0) -- ++(0,0,3) -- ++(0,-3,0) -- ++(1,0,0) -- ++(0,0,-3) -- cycle;
    
    \draw[line width=2pt, fill=blue!10] (-1,-11,3) -- ++(-8,0,0) -- ++(0,-3,0) -- ++(8,0,0) -- cycle;
    
    \draw[line width=2pt, fill=red!60] (-8,-11,3) -- ++(-1,0,0) -- ++(0,2,0) -- ++(0,0,-2) -- ++(1,0,0) -- ++(0,-2,0) -- cycle;
    
    \draw[line width=2pt, fill=red!10] (-8,-11,3) -- ++(7,0,0) -- ++(0,0,-2) -- ++(-7,0,0) -- cycle;
    
    \draw[line width=2pt, fill=brown!60] (0,-10,0) -- ++(-1,0,0) -- ++(0,0,2) -- ++(1,0,0) -- cycle;
    
    \draw[line width=2pt, fill=teal!60] (0,0,7) -- ++(1,0,0) -- ++(0,-3,0) -- ++ (0,0,02) -- ++(-1,0,0) -- ++(0,3,0) -- cycle;
    
    \draw[line width=2pt, fill=teal!10] (0,-4,6) -- ++(5,0,0) -- ++(0,0,2) -- ++(-5,0,0) -- cycle;
    
    \hex{5}{5}{3}{4}
    \dentB{5}{5}
    \dentB{5}{6}
    \dentC{5}{2}
    \dentC{5}{5}
    \end{scope}
\end{tikzpicture}

\begin{multicols}{2}

(a)

(b)
\end{multicols}
\end{center}
\caption{$H_{5,5,3,4,(5,6),(2,5)}$(a) partitioned into three tileable regions; (b) partitioned into many regions.} 
    \label{fig:fig7}
\end{figure}
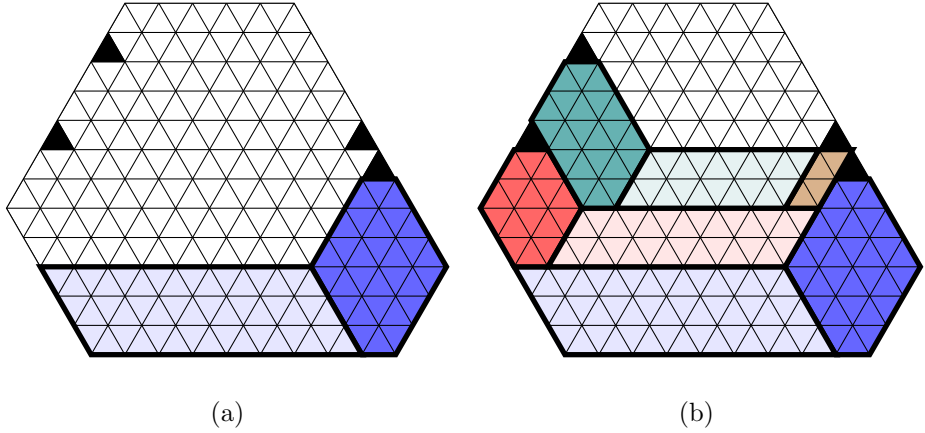

    The fact that $\mu_{t-1} \leq t-1$ implies that $u_m \geq t$, and $\underline{u_m} \leq b$. This guarantees the northeast and southwest sides of the dark hexagon are of length $\leq b$, so that the light region west of the dark hexagon is actually a parallelogram. Had it been the case that $v_n > u_m$, we would simply flip the picture.
\end{proof}

Our application of the inductive hypothesis in the proof implies that one could continue subdividing the entire region into parallelograms and semiregular hexagons. Such a subdivision is depicted in Figure \ref{fig:fig7} (b). \thref{existence} shows that untileable dented hexagons are characterized by having too many dents too far north, which yields the following consequence.

\begin{corollary}\thlabel{existcor}
    Let $H=H_{a,b,c,n+m,(u_i)_1^m, (v_j)_1^n}$ and $H'=H_{a',b',c',n+m,(u'_i)_1^m, (v'_j)_1^n}$ be dented hexagons with $u_i \leq u_i'$ and $v_j \leq v_j'$ for each $i, j$. If $H$ has a tiling, so does $H'$. \qed
\end{corollary}

Our main result is that families of dented hexagons with fixed parameters, $b,c, t, \vec u, \vec v$ have a tiling function given by the following rational function in $a$ (in fact, a polynomial in $a$):

\begin{theorem}\thlabel{main}
 Where $H_{a,b,c,t,(u_i)_1^m, (v_j)_1^n}$ is a tileable dented hexagon (implying $t=n+m$),
 \begin{align*}
     {\M(H_{a,b,c,t,\vec u, \vec v})} &= {\M(H_{0,b,c,t,\vec u, \vec v})} \prod_{i=1}^m(u_i)_{\underline{u_i}} \prod_{j=1}^n(v_j)_{\underline{v_j}}\\
     & \times \frac{\Pl(a,b+n, c+m)}{\prod_{i=1}^m(a+u_i)_{\underline{u_i}}\prod_{j=1}^n(a+v_j)_{\underline{v_j}}}.
 \end{align*}
\end{theorem}


When $H_{0,b,c,t,\vec u, \vec v}$ has a nice tiling function, \thref{main} shows that $H_{a,b,c,t, \vec u, \vec v}$ also has a nice tiling function. In particular, this explains a result found by Lai (Theorem 3.1 from \cite{La17}, with $q=1$), that when the dents along each side of the region are adjacent and ${v_n} = {u_m}$ the tiling function for the region is relatively simple. We discuss this and related results in section \ref{corollaries}.

\section{Groundwork} \label{groundwork}

In this section we lay the technical foundation for the arguments that follow.

\begin{lemma}
    If $H_{0,b,c,t,\vec u, \vec v}$ is a tileable dented hexagon, then $\M(H_{a,b,c,t,\vec u, \vec v})$ is a polynomial in $a$ when fixing $b,c,t$ and the vectors $\vec u$ and $\vec v$.
\end{lemma}

\begin{proof}
    The tilings of $H_{0,b,c,t,(u_i)_1^m, (v_j)_1^n}$ are in bijection with families of $(b+n)$ non-intersecting lattice paths which start along the southwest boundaries of dented hexagon and end at the northeast boundaries. This bijection is explained briefly in Figure \ref{fig:lattice}. Using this type of bijection is a standard technique when enumerating lozenge tilings, and similar examples of the technique explained in greater detail can be found in \cite{Ci01}.

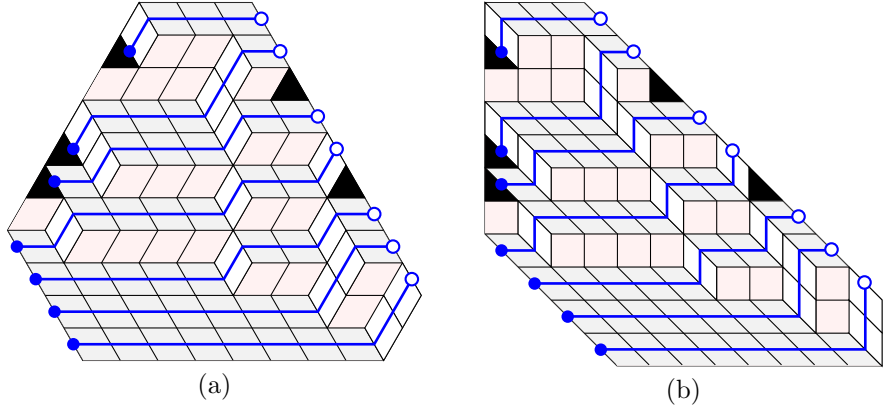
\begin{figure}
\begin{minipage}[c]{\textwidth}
\begin{center}
\begin{multicols}{2}

\begin{tikzpicture}[x={(0:1cm)},y={(120:1cm)},z={(240:1cm)},scale=.5]
    
    \draw[] (0,0,3) --++(0,0,7) --++(0,-4,0) --++(8,0,0) --++(0,0,-2) -- ++(0,9,0) -- cycle;
    
    \begin{scope}
        \clip (0,0,3) --++(0,0,7) --++(0,-4,0) --++(8,0,0) --++(0,0,-2) -- ++(0,9,0) -- cycle;
        \fill[red!5] (0,0,3) --++(0,0,7) --++(0,-4,0) --++(8,0,0) --++(0,0,-2) -- ++(0,9,0) -- cycle;
        
            \foreach \x in {0,1,2,3}{\draw(\x,0,10+\x) --++(10,0,0);}
            \foreach \x in {0,1,2,3,4,5,6}{\draw(0,0,3+\x) --++(3+\x,0,0);}
            \foreach \x in{1,2,3,4,5,6,7,8,9,10}{
            \draw (\x,0,3) --++(0,0,10);
            }
    \end{scope}

        
    \begin{scope}
     \clip(0,0,3) --++ (0,-1,0) --++(2,0,0) --++ (0,0,2) --++(-3,0,0) --++(0,-2,0) --++(3,0,0) --++(0,0,-1)--++(2,0,0) --++(0,0,1) --++(-2,0,0) --++(0,0,1) --++(-3,0,0) --++(0,1,0) --++(-1,0,0) --++(0,-2,0) --++(4,0,0) --++(0,0,-1) --++(2,0,0) --++(0,0,1) --++(-2,0,0)--++(0,-1,0)--++(4,0,0) --++(0,0,2) --++(-1,0,0) --++(0,1,0) --++(-2,0,0) --++(0,1,0)--++(-4,0,0)--++(0,1,0)--++(-1,0,0)--++(0,-4,0)--++(8,0,0)--++(0,0,-2) --++ (0,6,0) --++(-2,0,0)--++(0,0,-1)--++(1,0,0)--++(0,2,0)--++(-3,0,0);
    \fill[black!5](0,0,3) --++ (0,-1,0) --++(2,0,0) --++ (0,0,2) --++(-3,0,0) --++(0,-2,0) --++(3,0,0) --++(0,0,-1)--++(2,0,0) --++(0,0,1) --++(-2,0,0) --++(0,0,1) --++(-3,0,0) --++(0,1,0) --++(-1,0,0) --++(0,-2,0) --++(4,0,0) --++(0,0,-1) --++(2,0,0) --++(0,0,1) --++(-2,0,0)--++(0,-1,0)--++(4,0,0) --++(0,0,2) --++(-1,0,0) --++(0,1,0) --++(-2,0,0) --++(0,1,0)--++(-4,0,0)--++(0,1,0)--++(-1,0,0)--++(0,-4,0)--++(8,0,0)--++(0,0,-2) --++ (0,6,0) --++(-2,0,0)--++(0,0,-1)--++(1,0,0)--++(0,2,0)--++(-3,0,0);
    
    \foreach \x in {0,1,2,3}{\draw(\x,0,10+\x) --++(10,0,0);}
    \foreach \x in {0,1,2,3,4,5,6}{\draw(0,0,3+\x) --++(3+\x,0,0);}
    \foreach \x in {0,...,13}{
    \draw(-10+\x,-3,0) --++ (0,-11,0);
    }
    \end{scope}
    
    \vdent 003 \vdent006 \vdent 107 \vdent 109 \vdent 304
    \vdent 305 \vdent 406 \vdent 707 \vdent 508 \vdent 809
    \vdent 60{10} \vdent 90{10} \vdent 90{11} \vdent {11}0{11}  \vdent {11}0{12}
    
    \draw[line width=1pt, blue] (0,-.5,4) -- ++(0,0,-1) -- ++(3,0,0);
    
    \draw[line width=1pt, blue] (0,-.5,7) -- ++(0,0,-1) -- ++(3,0,0) -- ++(0,0,-2) -- ++(1,0,0);
    
    \draw[line width=1pt, blue] (0,-.5,8) --++(1,0,0)-- ++(0,0,-1) -- ++(3,0,0) -- ++(0,0,-1) -- ++(2,0,0);
    
    \draw[line width=1pt, blue] (0,-.5,10) --++(1,0,0)-- ++(0,0,-1) -- ++(4,0,0) -- ++(0,0,-1) -- ++(2,0,0) -- ++(0,0,-1);
    
    \draw[line width=1pt, blue] (0,-1.5,10) --++(5,0,0)-- ++(0,0,-1) -- ++(2,0,0) -- ++(0,0,-1) -- ++(1,0,0);
    
    \draw[line width=1pt, blue] (0,-2.5,10) --++(7,0,0)-- ++(0,0,-2) -- ++(1,0,0);
    
    \draw[line width=1pt, blue] (0,-3.5,10) --++(8,0,0)-- ++(0,0,-2);
    
    \dentB{3}{3}
    \dentB{3}{6}
    \dentC{3}{2}
    \dentC{3}{5}
    \dentC{3}{6}
    
    \draw   node[circle,fill=blue,scale=.5] at (0,-.5,4) {}
            node[circle,fill=blue,scale=.5] at (0,-.5,7) {}
            node[circle,fill=blue,scale=.5] at (0,-.5,8) {}
            node[circle,fill=blue,scale=.5] at (0,-.5,10) {}
            node[circle,fill=blue,scale=.5] at (1,-.5,11) {}
            node[circle,fill=blue,scale=.5] at (2,-.5,12) {}
            node[circle,fill=blue,scale=.5] at (3,-.5,13) {};
            
    \draw   node[circle,thick,draw,blue,fill=white,scale=.5] at (0,-3.5,0){}
            node[circle,thick,draw,blue,fill=white,scale=.5] at (0,-4.5,0){}
            node[circle,thick,draw,blue,fill=white,scale=.5] at (0,-6.5,0){}
            node[circle,thick,draw,blue,fill=white,scale=.5] at (0,-7.5,0){}
            node[circle,thick,draw,blue,fill=white,scale=.5] at (0,-9.5,0){}
            node[circle,thick,draw,blue,fill=white,scale=.5] at (0,-10.5,0){}
            node[circle,thick,draw,blue,fill=white,scale=.5] at (0,-11.5,0){};

\end{tikzpicture}

(a)

\vfill \null \columnbreak

\begin{tikzpicture}[x={(0:1cm)},y={(135:1.41cm)},z={(270:1cm)},scale=.44]
    
    \draw[] (0,0,3) --++(0,0,7) --++(0,-4,0) --++(8,0,0) --++(0,0,-2) -- ++(0,9,0) -- cycle;
    
    \begin{scope}
        \clip (0,0,3) --++(0,0,7) --++(0,-4,0) --++(8,0,0) --++(0,0,-2) -- ++(0,9,0) -- cycle;
        \fill[red!5] (0,0,3) --++(0,0,7) --++(0,-4,0) --++(8,0,0) --++(0,0,-2) -- ++(0,9,0) -- cycle;
        
            \foreach \x in {0,1,2,3}{\draw(\x,0,10+\x) --++(10,0,0);}
            \foreach \x in {0,1,2,3,4,5,6}{\draw(0,0,3+\x) --++(3+\x,0,0);}
            \foreach \x in{1,2,3,4,5,6,7,8,9,10}{
            \draw (\x,0,3) --++(0,0,10);
            }
    \end{scope}

        
    \begin{scope}
     \clip(0,0,3) --++ (0,-1,0) --++(2,0,0) --++ (0,0,2) --++(-3,0,0) --++(0,-2,0) --++(3,0,0) --++(0,0,-1)--++(2,0,0) --++(0,0,1) --++(-2,0,0) --++(0,0,1) --++(-3,0,0) --++(0,1,0) --++(-1,0,0) --++(0,-2,0) --++(4,0,0) --++(0,0,-1) --++(2,0,0) --++(0,0,1) --++(-2,0,0)--++(0,-1,0)--++(4,0,0) --++(0,0,2) --++(-1,0,0) --++(0,1,0) --++(-2,0,0) --++(0,1,0)--++(-4,0,0)--++(0,1,0)--++(-1,0,0)--++(0,-4,0)--++(8,0,0)--++(0,0,-2) --++ (0,6,0) --++(-2,0,0)--++(0,0,-1)--++(1,0,0)--++(0,2,0)--++(-3,0,0);
    \fill[black!5](0,0,3) --++ (0,-1,0) --++(2,0,0) --++ (0,0,2) --++(-3,0,0) --++(0,-2,0) --++(3,0,0) --++(0,0,-1)--++(2,0,0) --++(0,0,1) --++(-2,0,0) --++(0,0,1) --++(-3,0,0) --++(0,1,0) --++(-1,0,0) --++(0,-2,0) --++(4,0,0) --++(0,0,-1) --++(2,0,0) --++(0,0,1) --++(-2,0,0)--++(0,-1,0)--++(4,0,0) --++(0,0,2) --++(-1,0,0) --++(0,1,0) --++(-2,0,0) --++(0,1,0)--++(-4,0,0)--++(0,1,0)--++(-1,0,0)--++(0,-4,0)--++(8,0,0)--++(0,0,-2) --++ (0,6,0) --++(-2,0,0)--++(0,0,-1)--++(1,0,0)--++(0,2,0)--++(-3,0,0);
    
    \foreach \x in {0,1,2,3}{\draw(\x,0,10+\x) --++(10,0,0);}
    \foreach \x in {0,1,2,3,4,5,6}{\draw(0,0,3+\x) --++(3+\x,0,0);}
    \foreach \x in {0,...,13}{
    \draw(-10+\x,-3,0) --++ (0,-11,0);
    }
    \end{scope}
    
    \vdent 003 \vdent006 \vdent 107 \vdent 109 \vdent 304
    \vdent 305 \vdent 406 \vdent 707 \vdent 508 \vdent 809
    \vdent 60{10} \vdent 90{10} \vdent 90{11} \vdent {11}0{11}  \vdent {11}0{12}
    
    \draw[line width=1pt, blue] (0,-.5,4) -- ++(0,0,-1) -- ++(3,0,0);
    
    \draw[line width=1pt, blue] (0,-.5,7) -- ++(0,0,-1) -- ++(3,0,0) -- ++(0,0,-2) -- ++(1,0,0);
    
    \draw[line width=1pt, blue] (0,-.5,8) --++(1,0,0)-- ++(0,0,-1) -- ++(3,0,0) -- ++(0,0,-1) -- ++(2,0,0);
    
    \draw[line width=1pt, blue] (0,-.5,10) --++(1,0,0)-- ++(0,0,-1) -- ++(4,0,0) -- ++(0,0,-1) -- ++(2,0,0) -- ++(0,0,-1);
    
    \draw[line width=1pt, blue] (0,-1.5,10) --++(5,0,0)-- ++(0,0,-1) -- ++(2,0,0) -- ++(0,0,-1) -- ++(1,0,0);
    
    \draw[line width=1pt, blue] (0,-2.5,10) --++(7,0,0)-- ++(0,0,-2) -- ++(1,0,0);
    
    \draw[line width=1pt, blue] (0,-3.5,10) --++(8,0,0)-- ++(0,0,-2);
    
    \dentB{3}{3}
    \dentB{3}{6}
    \dentC{3}{2}
    \dentC{3}{5}
    \dentC{3}{6}
    
    \draw   node[circle,fill=blue,scale=.5] at (0,-.5,4) {}
            node[circle,fill=blue,scale=.5] at (0,-.5,7) {}
            node[circle,fill=blue,scale=.5] at (0,-.5,8) {}
            node[circle,fill=blue,scale=.5] at (0,-.5,10) {}
            node[circle,fill=blue,scale=.5] at (1,-.5,11) {}
            node[circle,fill=blue,scale=.5] at (2,-.5,12) {}
            node[circle,fill=blue,scale=.5] at (3,-.5,13) {};
            
    \draw   node[circle,thick,draw,blue,fill=white,scale=.5] at (0,-3.5,0){}
            node[circle,thick,draw,blue,fill=white,scale=.5] at (0,-4.5,0){}
            node[circle,thick,draw,blue,fill=white,scale=.5] at (0,-6.5,0){}
            node[circle,thick,draw,blue,fill=white,scale=.5] at (0,-7.5,0){}
            node[circle,thick,draw,blue,fill=white,scale=.5] at (0,-9.5,0){}
            node[circle,thick,draw,blue,fill=white,scale=.5] at (0,-10.5,0){}
            node[circle,thick,draw,blue,fill=white,scale=.5] at (0,-11.5,0){};

\end{tikzpicture}

(b)
\end{multicols}
\end{center}
\caption{Tilings of $H_{3,4,2,5,(3,6),(2,5,6)}$ correspond to lattice paths which start on edges along the southwest borders of the region (including along the southwest side of the hexagon and also along the western dents) and end at edges along the northeast borders, the steps of which are directly east or northeast. The correspondence is as follows: given a tiling, place dots in the middle of these edges as exhibited in (a). Then, connect each dot by a path to the middle of the opposite edge of the lozenge containing it, and then to the middle of the opposite edge of the next lozenge, and so on until the dots are connected by paths. Warping the picture as in (b) transforms the paths into non-intersecting Up-Right lattice paths. This process is reversible.} 
    \label{fig:lattice}
\end{minipage}
\end{figure}
    
    The ``start points'' along the southwest side of the region can be interpreted as having coordinates $\{(-i,i): i \in [b]\}$, and the start points from the dents would then have coordinates $\{(-b,b+c+t+1-{v_i}): i\in [n]\}$. The ``end points'' then have coordinates $\{(a-b-1+j,b+c+t+1-j): j \in [b+t] - \{u_i: i \in [m]\}\}$. It is clear these satisfy the path-intersection condition of \thref{GV} when each set is labeled with indices increasing from north to south.
    
    Applying \thref{GV}, the number of lattice paths with these start and end points is given by the determinant of a matrix, the entries of which are of the form ${a-1+v_i \choose v_i-j}$ and ${a+c+t\choose b+c+t+1-j-i}$. Since each entry in the matrix is a polynomial on $a$, so is the determinant. 

\end{proof}
The arguments that follow are largely about the proportionality of polynomials. For $p, q$ polynomials in $a$ with we will use the non-standard notation
$$p(a) \aequiv q(a)$$
to mean there is a nonzero constant $c$ with $c\cdot p(a)=q(a)$. We will sometimes use this notation where $p,q$ are functions over many parameters which are polynomials over $a$ when all other parameters are fixed, in which case $c$ may depend on parameters other than $a$.
It is easy to check $\aequiv$ is an equivalence relation and satisfies the following properties, where $p_1,p_2,$ and $q$ are polynomials in $a$, and $q$ is not identically 0.
\begin{align}
    \brac{p_1(a)} \aequiv \brac{p_2(a)} & \Rightarrow \brac{p_1(a)+p_2(a)} \aequiv \brac{p_1(a)} \aequiv \brac{p_2(a)} \label{eq:sum}\\
    \brac{p_1(a)} \aequiv \brac{p_2(a)} & \iff \brac{p_1(a)q(a)} \aequiv \brac{p_2(a) q(a)} \label{eq:substitute}
\end{align}

We will make use of the following technical lemma.
\begin{lemma}\thlabel{sub}
    For $d \in \mathbb Z$ and $y \in \mathbb Z^+$ and $z \in \mathbb N$,
    \begin{equation} \label{eq:sub}
    \brac{\Pl(a+d,y-1, z+1)} \aequiv \brac{\Pl(a+d,y,z) \frac{(a+d+z+1)_{y-1}}{(a+d+y)_z}}.
    \end{equation}
\end{lemma}

\begin{proof}
    It is straightforward to show from the definition of $\Pl(x,y,z)$ that $\frac{\Pl(x,y,z+1)}{\Pl(x,y,z)} = \frac{(x+z+1)_y}{(z+1)_y}$; in turn it is straightforward to show
    \begin{equation}\label{eq:subgen}
    \Pl(x,y-1,z+1) = \Pl(x,y,z)\frac{(y)_{z}(x+z+1)_{y-1}}{(x+y)_{z}(z+1)_{y-1}}.
    \end{equation}
Equation \eqref{eq:sub} follows from equation \eqref{eq:subgen} where $x=a+d$.
\end{proof}

\section{Proof of \thref{main}}

For $b,c \in \mathbb N$ and $\vec u = (u_i)_1^m, \vec v = (v_j)_1^n$ vectors over $\mathbb N$ so that there exists a well-defined region $H_{0,b,c,m+n, \vec u, \vec v}$, define
$$f_{b,c,\vec u, \vec v}(a) := \frac{\Pl(a,b+n, c+m)}{\prod_{i=1}^m(a+u_i)_{\underline{u_i}}\prod_{j=1}^n(a+v_j)_{\underline{v_j}}}.$$
The main method of our proof is to show that $\M(H_{a,b,c,m+n,\vec u, \vec v})$ interpreted as a function in $a$ is proportional to $f$.
We will do this by showing $f$ is a polynomial in $a$ which can be defined recursively and obeys the same recursion as the tiling function. Observe that the equation in \thref{main} can be expressed
$$\brac{\M(H_{a,b,c,m+n,\vec u, \vec v})} \aequiv \brac{f_{b,c,\vec u, \vec v}(a)}.$$

It will sometimes be convenient for us to assume that $u_1 >1$ or $v_1>1$. A dented hexagon with $u_1=v_1=1$ has no tilings by \thref{existence}, in which case \thref{main} holds by vacuously. In the case that $u_1 = 1$, and $v_1> 1$ or $n=0$, the top row of triangles of the region $H_{a,b,c,t,\vec u, \vec v}$ are covered by forced lozenges, and may be removed from the region. Figure \ref{fig:fig4} (a) depicts an example of this. Removing the forced lozenges gives the region $H_{a+1,b,c,t-1,(u_{i}-1)_{i=2}^m, (v_j -1)_{j=1}^n}$, which thus has the same tiling function as $H_{a,b,c,t,\vec u, \vec v}$. We should like to show that our definition of $f_{b,c,\vec u, \vec v}$ respects this and the analogous case with $v_1=1.$

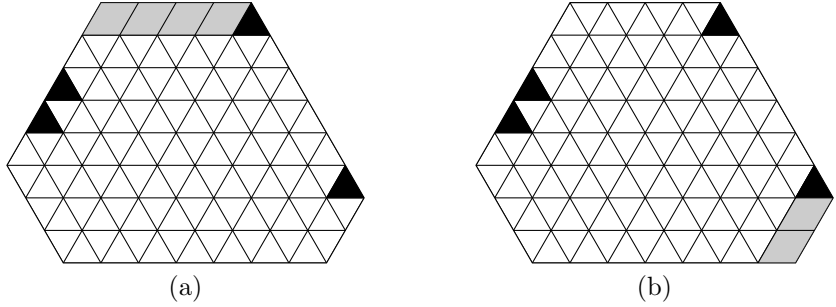
\begin{figure}
\begin{center}
\begin{multicols}{2}

\begin{tikzpicture}[x={(0:1cm)},y={(120:1cm)},z={(240:1cm)},scale=.5]
    
    \hex{4}{3}{2}{3}
    \dentB{4}{1}
    \dentB{4}{6}
    \dentC{4}{3}
    \dentC{4}{4}
    
    \draw[fill=black!20] (0,-4,0) -- ++(-4,0,0) -- ++(0,0,1) -- ++(4,0,0) -- cycle;
    
    \draw (0,-3,1) -- ++(0,0,1);
    \draw (-1,-3,1) -- ++(0,0,1);
    \draw (-2,-3,1) -- ++(0,0,1);
    
\end{tikzpicture}

(a)

\vfill \null \columnbreak

\begin{tikzpicture}[x={(0:1cm)},y={(120:1cm)},z={(240:1cm)},scale=.5]
    
    \hex{4}{3}{2}{3}    
    \draw[fill=black!20] (0,-10,0) -- ++(0,0,2) -- ++(-1,0,0) -- ++(0,0,-3) -- cycle;
    
    \draw (0,-10,1)--++(-1,0,0);
    \dentB{4}{1}
    \dentB{4}{6}
    \dentC{4}{3}
    \dentC{4}{4}

\end{tikzpicture}

(b)
\end{multicols}
\end{center}
\caption{$H_{4,3,2,3,(1,6),(3,4)}$ (a) the top row is forced since $u_1=1$; (b) the southeast side is forced since $\underline{u_m}=0$.} 
    \label{fig:fig4}
\end{figure}

If $\underline{u_m}=0$ (or $\underline{v_n}=0$), then the southeast (or southwest) side of the region is covered by forced lozenges. 
If $\underline{u_m}=0$, then $H_{a,b,c,t,\vec u, \vec v}$ thus has the same tiling function as $H_{a,b,c+1,t-1,(u_i)_{i=1}^{m-1}, \vec v}$, which is the region that is obtained by removing the forced tiles. This is depicted in Figure \ref{fig:fig4} (b). We wish to show our definition of $f_{b,c, \vec u, \vec v}$ respects this fact, and the analogous fact when $\underline{v_n}=0$.

\begin{lemma}\thlabel{loptop} If $H_{a,b, c, t,(u_i)_1^m, (v_j)_1^n}$ is tileable then
    \begin{enumerate}[label=\alph*.]
        \item If $u_1 = 1$, and $v_1>1$ or $n=0$, then
        $\brac{f_{b,c,\vec u, \vec v}(a)} \aequiv \brac{f_{b,c,(u_i-1)_2^m, (v_j-1)_1^n}(a+1)}.$
        \item If $v_1=1$, and $u_1>1$ or $m=0$, then
        $\brac{f_{b,c,\vec u, \vec v}(a)} \aequiv \brac{f_{b,c,(u_i-1)_1^m, (v_j-1)_2^n}(a+1)}.$
        \item If $\underline{u_m} = 0$ then
        $f_{b,c,\vec u, \vec v}(a) = f_{b,c+1,(u_i)_{i=1}^{m-1}, (v_j)_1^n}(a).$
        \item If $\underline{v_n} = 0$ then
        $f_{b,c,\vec u, \vec v}(a) = f_{b+1,c,(u_i)_{i=1}^{m}, (v_j)_1^{n-1}}(a).$ 
    \end{enumerate}
\end{lemma}
\begin{proof}
These identities are straightforward to check when written explicitly.
\end{proof}


Let $\vec \emptyset$ denote the empty vector. We will now verify that \thref{main} holds when $\vec u = \vec \emptyset$. The case where $\vec v = \vec \emptyset$ follows by symmetry. 

\begin{lemma}\thlabel{CLP}
    The dented hexagon $H_{a,b,c,n,\vec \emptyset, (v_j)_1^n}$ is tileable and
    \begin{equation}\label{eq:eqCLP}
    \brac{\M(H_{a,b,c,n,\vec \emptyset, (v_j)_1^n})} \aequiv f_{b,c, \vec \emptyset, \vec v}(a) := \brac{ \frac{\Pl(a,b+n,c)}{\prod_{j=1}^n (a+v_j)_{\underline{v_j}}}}.
    \end{equation}
\end{lemma}

This is equivalent to a result by Cohn, Larsen, and Propp \cite{CLP}, but the equivalence is not obvious so we will prove this result directly using Kuo Condensation.

\begin{proof}
Dented hexagons with dents on just one side are tileable by \thref{existence}.

We will induct on $n+c$, with base cases at $n=0$ and $c=0$. Note that if $c=0$ then the region has a unique tiling\footnote{The entire northwest side is comprised of dents, and the unique tiling extends to the unique tiling of the $a \times b$ parallelogram.} and $f_{b,0, \vec \emptyset, \vec v}(a):= \frac{\Pl(a,b+n,0)}{\prod_{j=1}^n(a+v_j)_{0}} = 1.$ In the case $n=0$, equation \eqref{eq:eqCLP} reduces to MacMahon's formula \eqref{eq:MacMahon}.

For our inductive hypothesis, assume that \eqref{eq:eqCLP} holds for dented hexagons $H_{a,b',c',n',\vec \emptyset, (v_j')_1^{n-1}}$ with $c'+n' < c+n$. Consider dented hexagons $H_{a,b,c,n,\vec \emptyset, (v_j)_1^n}$. 

In the case $v_1=1$, $H_{a,b,c,n,\vec \emptyset, (v_j)_1^n}$ has forced lozenges along its northern side so has the same tiling function as $H_{a+1,b,c,n-1,\vec \emptyset, (v_j-1)_2^{n}}$, so
\begin{align*}
    \brac{\M\left(H_{a,b,c,n,\vec \emptyset, (v_j)_1^n}\right)}&\aequiv
    \brac{\M\left(H_{a+1,b,c,n-1,\vec \emptyset, (v_j-1)_2^{n}}\right)}\\
    (\mbox{by IH}) &\aequiv \brac{f_{b,c,\vec \emptyset, (v_j-1)_2^{n}}(a+1)}\\
    (\mbox{by \thref{loptop}b}) &\aequiv \brac{ f_{b,c,\vec \emptyset, (v_j)_1^{n}}(a)}.
\end{align*}

In the case $\underline{v_n}=0$, $H_{a,b,c,n,\vec \emptyset, (v_j)_1^n}$ has forced lozenges along its southwest side so has the same tiling function as $H_{a,b+1,c,n-1,\vec \emptyset, (v_j)_1^{n-1}}$, so
\begin{align*}
    \brac{\M\left(H_{a,b,c,n,\vec \emptyset, (v_j)_1^n}\right)}&\aequiv
    \brac{\M\left(H_{a,b+1,c,n-1,\vec \emptyset, (v_j)_1^{n-1}}\right)}\\
    (\mbox{by IH}) &\aequiv \brac{f_{b,c+1,\vec \emptyset, (v_j)_1^{n-1}}(a)}\\
    (\mbox{by \thref{loptop}d}) &\aequiv \brac{ f_{b,c,\vec \emptyset, (v_j)_1^{n}}(a)}.
\end{align*}

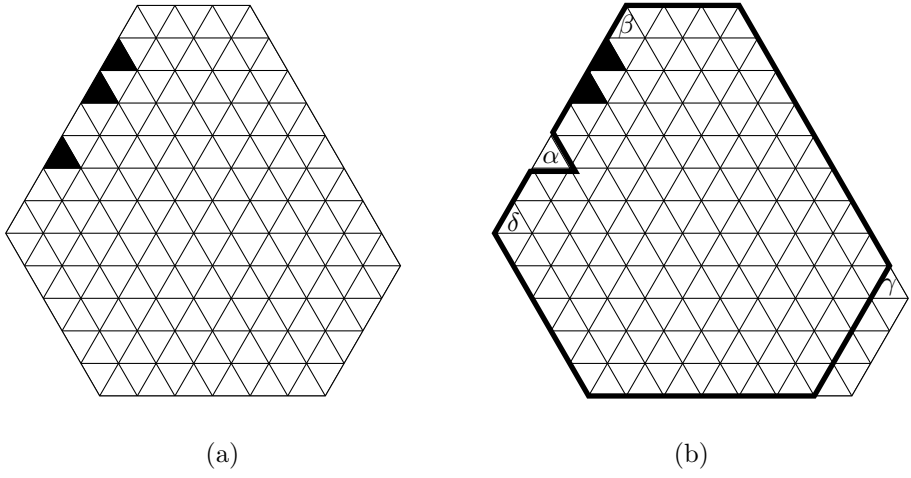
\begin{figure}
\begin{minipage}[c]{\textwidth}
\begin{center}

\begin{tikzpicture}[x={(0:1cm)},y={(120:1cm)},z={(240:1cm)},scale=.5]
    \begin{scope}
    \hex{3}{5}{4}{3}
    \dentC{3}{2}
    \dentC{3}{3}
    \dentC{3}{5}
    \end{scope}
    
    \begin{scope}[xshift=13cm]
    
    \draw[line width=2pt] (0,0,3)-- ++(0,0,3.9) --++(0,-1.2,0) --++(-1.2,0,0) --++(0,0,1.9) -- ++(0,-5,0) -- ++(6,0,0) -- ++(0,0,-4) -- ++(0,8,0) -- cycle;

    \node (A) at (0,-.33,7.33){$\alpha$};
    \node (B) at (0,-.33,3.33){$\beta$};
    \node (C) at (0,-11.33,.33){$\gamma$};
    \node (D) at (0,-.33,9.33){$\delta$};
    
    \hex{3}{5}{3}{4}
    \dentC{3}{2}
    \dentC{3}{3}
    \end{scope}
    
\end{tikzpicture}

\begin{multicols}{2}

(a)

(b)
\end{multicols}
\end{center}
\caption{(a) $H_{3,5,4,3,\vec \emptyset,(2,3,5)}$; (b) $R_3$ with $\alpha, \beta, \gamma, \delta$ labeled, and the border of $H_{3,5,4,3,\vec \emptyset, (2,3,5)}$ depicted with a thick line.} 
    \label{fig:fig8}
\end{minipage}
\end{figure}

We therefore assume that $n>0$, $v_1>1$, and $\underline{v_n}>0$. Regard $H_{a,b,c,n,\vec \emptyset, \vec v}$ as a subregion of the (unbalanced) region $R_a=H_{a, b, c-1, n+1, \vec \emptyset, (v_j+1)_{j=1}^{n-1}}$. 
Within each region $R_a$, let $\alpha$ be the unit triangle indexed by $v_n$, let $\beta$ be the northmost triangle on the northwest side, let $\gamma$ be the southmost triangle along the northeast side, and let $\delta$ be the southmost triangle along the northwest side. These placements are depicted in Figure \ref{fig:fig8}.

We will apply Kuo Condensation (\thref{kuo}) to $R_a, \alpha, \beta, \gamma, \delta$.
Figure \ref{fig:fig9}  depicts each region referenced in the condensation formula, in some cases with forced lozenges shaded (we regard these as not being part of the region). Note they are all balanced dented hexagons.
Since these regions can be regarded as dented hexagons with dents on only one side, each is tileable by \thref{existence}. The parameter $c+n$ is strictly minimal on $R_a - \alpha - \gamma$, so we may apply the inductive hypothesis to each of the other regions.

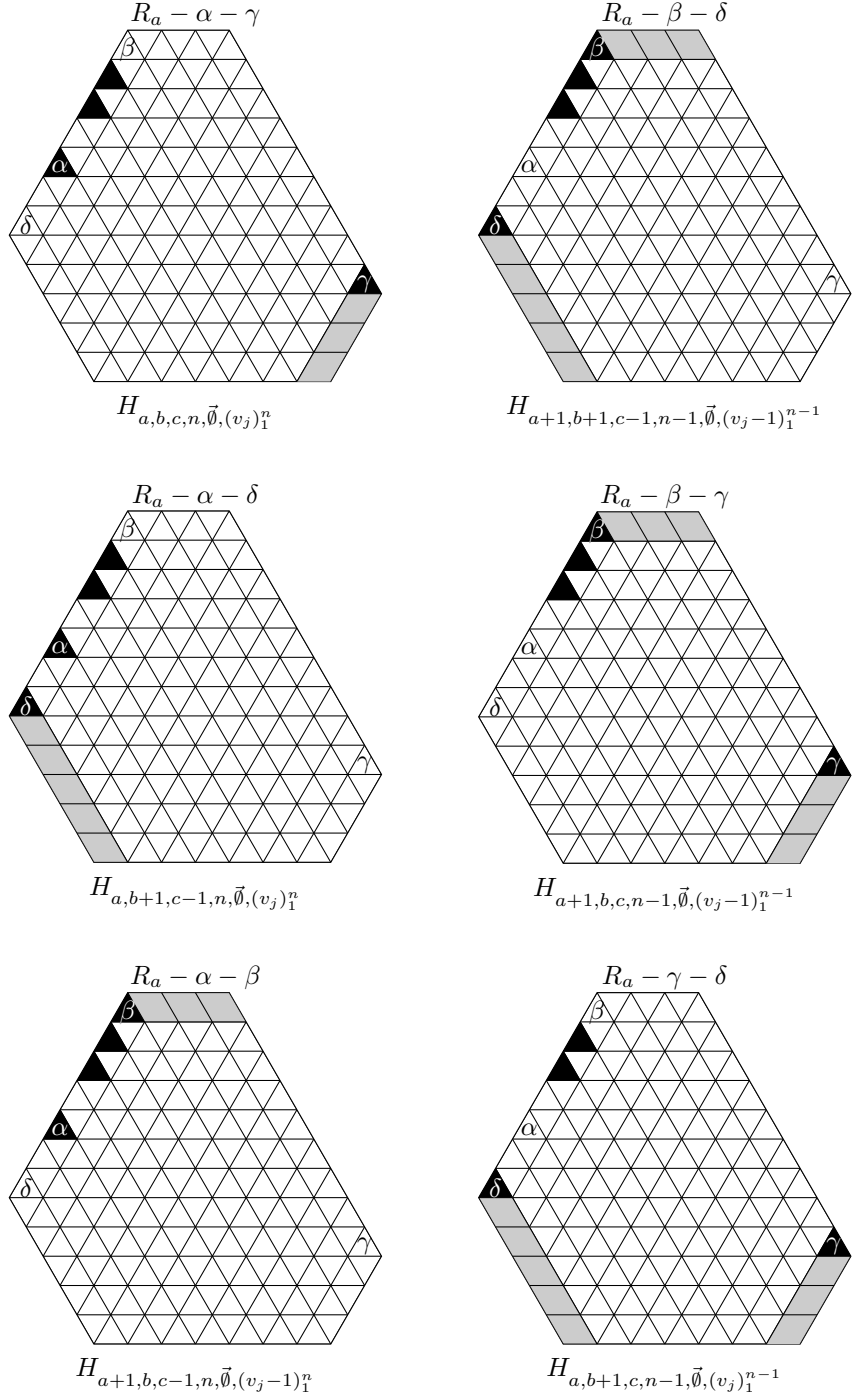
\begin{figure} \vspace{-1cm}
\begin{minipage}[c]{\textwidth}
\begin{center}
\begin{multicols}{2}

$R_a - \alpha - \gamma$

\begin{tikzpicture}[x={(0:1cm)},y={(120:1cm)},z={(240:1cm)},scale=.45]

    \node (A) at (0,-.33,7.33){$\alpha$};
    \node (B) at (0,-.33,3.33){$\beta$};
    \node (C) at (0,-11.33,.33){$\gamma$};
    \node (D) at (0,-.33,9.33){$\delta$};
    
    \hex{3}{5}{3}{4}
    \dentC{3}{2}
    \dentC{3}{3}
    
    
    \dentC{3}{5}
    \node (A) at (0,-.33,7.33){\color{white}$\alpha$};
    
    
    \dentB{3}{9} \draw[fill=black!20] (0,-12,0) -- ++(0,0,4) -- ++(-1,0,0) -- ++(0,0,-4) -- cycle;
    \node (C) at (0,-11.33,.33){\color{white}$\gamma$};
    \draw (0,-12,1) -- ++(-1,0,0);
    \draw (0,-12,2) -- ++(-1,0,0);
    
    
\end{tikzpicture}

$H_{a,b,c,n, \vec \emptyset, (v_j)_1^n}$

\vfill \null \columnbreak

$R_a - \beta - \delta$

\begin{tikzpicture}[x={(0:1cm)},y={(120:1cm)},z={(240:1cm)},scale=.45]

    \node (A) at (0,-.33,7.33){$\alpha$};
    \node (B) at (0,-.33,3.33){$\beta$};
    \node (C) at (0,-11.33,.33){$\gamma$};
    \node (D) at (0,-.33,9.33){$\delta$};
    
    \hex{3}{5}{3}{4}
    \dentC{3}{2}
    \dentC{3}{3}
    
    
    
    \dentC{3}{1} \draw[fill=black!20] (0,0,3) -- ++(0,-1,0) -- ++(3,0,0) -- ++(0,1,0) -- cycle;
    \node (B) at (0,-.33,3.33){\color{white}$\beta$};
    \draw (1,0,3) --++(0,-1,0);
    \draw (2,0,3) --++(0,-1,0);
    
    
    \dentC{3}{7} \draw[fill=black!20] (0,0,10) -- ++(1,0,0) -- ++(0,-5,0) -- ++(-1,0,0) -- cycle;
    \node (D) at (0,-.33,9.33){\color{white}$\delta$};
    \draw (0,-1,10) --++(1,0,0);
    \draw (0,-2,10) --++(1,0,0);
    \draw (0,-3,10) --++(1,0,0);
    \draw (0,-4,10) --++(1,0,0);
    
\end{tikzpicture}

$H_{a+1,b+1,c-1,n-1, \vec \emptyset, (v_j-1)_1^{n-1}}$

\end{multicols}\vspace{-.5cm}
\begin{multicols}{2}

$R_a - \alpha - \delta$
\begin{tikzpicture}[x={(0:1cm)},y={(120:1cm)},z={(240:1cm)},scale=.45]

    \node (A) at (0,-.33,7.33){$\alpha$};
    \node (B) at (0,-.33,3.33){$\beta$};
    \node (C) at (0,-11.33,.33){$\gamma$};
    \node (D) at (0,-.33,9.33){$\delta$};
    
    \hex{3}{5}{3}{4}
    \dentC{3}{2}
    \dentC{3}{3}
    
    
    \dentC{3}{5}
    \node (A) at (0,-.33,7.33){\color{white}$\alpha$};
    
    
    
    \dentC{3}{7} \draw[fill=black!20] (0,0,10) -- ++(1,0,0) -- ++(0,-5,0) -- ++(-1,0,0) -- cycle;
    \node (D) at (0,-.33,9.33){\color{white}$\delta$};
    \draw (0,-1,10) --++(1,0,0);
    \draw (0,-2,10) --++(1,0,0);
    \draw (0,-3,10) --++(1,0,0);
    \draw (0,-4,10) --++(1,0,0);
    
\end{tikzpicture}
$H_{a,b+1,c-1,n,\vec \emptyset, (v_j)_1^n}$

\vfill \null \columnbreak

$R_a - \beta - \gamma$

\begin{tikzpicture}[x={(0:1cm)},y={(120:1cm)},z={(240:1cm)},scale=.45]

    \node (A) at (0,-.33,7.33){$\alpha$};
    \node (B) at (0,-.33,3.33){$\beta$};
    \node (C) at (0,-11.33,.33){$\gamma$};
    \node (D) at (0,-.33,9.33){$\delta$};
    
    \hex{3}{5}{3}{4}
    \dentC{3}{2}
    \dentC{3}{3}
    
    
    
    \dentC{3}{1} \draw[fill=black!20] (0,0,3) -- ++(0,-1,0) -- ++(3,0,0) -- ++(0,1,0) -- cycle;
    \node (B) at (0,-.33,3.33){\color{white}$\beta$};
    \draw (1,0,3) --++(0,-1,0);
    \draw (2,0,3) --++(0,-1,0);
    
    \dentB{3}{9} \draw[fill=black!20] (0,-12,0) -- ++(0,0,4) -- ++(-1,0,0) -- ++(0,0,-4) -- cycle;
    \node (C) at (0,-11.33,.33){\color{white}$\gamma$};
    \draw (0,-12,1) -- ++(-1,0,0);
    \draw (0,-12,2) -- ++(-1,0,0);
    
    
\end{tikzpicture}

$H_{a+1,b,c,n-1, \vec \emptyset, (v_j-1)_1^{n-1}}$

\end{multicols}\vspace{-.5cm}
\begin{multicols}{2}

$R_a - \alpha - \beta$

\begin{tikzpicture}[x={(0:1cm)},y={(120:1cm)},z={(240:1cm)},scale=.45]

    \node (A) at (0,-.33,7.33){$\alpha$};
    \node (B) at (0,-.33,3.33){$\beta$};
    \node (C) at (0,-11.33,.33){$\gamma$};
    \node (D) at (0,-.33,9.33){$\delta$};
    
    \hex{3}{5}{3}{4}
    \dentC{3}{2}
    \dentC{3}{3}
    
    
    \dentC{3}{5}
    \node (A) at (0,-.33,7.33){\color{white}$\alpha$};
    
    \dentC{3}{1} \draw[fill=black!20] (0,0,3) -- ++(0,-1,0) -- ++(3,0,0) -- ++(0,1,0) -- cycle;
    \node (B) at (0,-.33,3.33){\color{white}$\beta$};
    \draw (1,0,3) --++(0,-1,0);
    \draw (2,0,3) --++(0,-1,0);
    
    
    
\end{tikzpicture}

$H_{a+1,b,c-1,n,\vec \emptyset, (v_j-1)_1^n}$

\vfill \null \columnbreak

$R_a - \gamma - \delta$

\begin{tikzpicture}[x={(0:1cm)},y={(120:1cm)},z={(240:1cm)},scale=.45]

    \node (A) at (0,-.33,7.33){$\alpha$};
    \node (B) at (0,-.33,3.33){$\beta$};
    \node (C) at (0,-11.33,.33){$\gamma$};
    \node (D) at (0,-.33,9.33){$\delta$};
    
    \hex{3}{5}{3}{4}
    \dentC{3}{2}
    \dentC{3}{3}
    
    
    
    
    \dentB{3}{9} \draw[fill=black!20] (0,-12,0) -- ++(0,0,4) -- ++(-1,0,0) -- ++(0,0,-4) -- cycle;
    \node (C) at (0,-11.33,.33){\color{white}$\gamma$};
    \draw (0,-12,1) -- ++(-1,0,0);
    \draw (0,-12,2) -- ++(-1,0,0);
    
    \dentC{3}{7} \draw[fill=black!20] (0,0,10) -- ++(1,0,0) -- ++(0,-5,0) -- ++(-1,0,0) -- cycle;
    \node (D) at (0,-.33,9.33){\color{white}$\delta$};
    \draw (0,-1,10) --++(1,0,0);
    \draw (0,-2,10) --++(1,0,0);
    \draw (0,-3,10) --++(1,0,0);
    \draw (0,-4,10) --++(1,0,0);
    
\end{tikzpicture}

$H_{a,b+1,c,n-1,\vec \emptyset, (v_j)_1^{n-1}}$

\end{multicols} \vspace{-1cm}
\end{center}
\caption{$R_a$ with two of $\alpha, \beta, \gamma, \delta$ removed and forced lozenges shaded.} 
    \label{fig:fig9}
\end{minipage}
\end{figure}

\thref{kuo} can be expressed
$$M(R-\alpha-\gamma)M(R-\beta-\delta)=M(R-\alpha-\delta)M(R-\beta-\gamma)-M(R-\alpha-\beta)M(R-\gamma-\delta).$$
Removing forced lozenges as depicted in Figure \ref{fig:fig9}, we can rewrite this:
\begin{equation}
    \begin{array}{rl}
         &M\left(H_{a,b,c,n,\vec \emptyset, (v_j)_1^n}\right)
         M\left(H_{a+1,b+1,c-1,n-1, \vec \emptyset, (v_j-1)_1^{n-1}}\right)\\[10pt]
         =&
         M\left(H_{a,b+1,c-1,n,\vec \emptyset, (v_j)_1^n}\right)
         M\left(H_{a+1,b,c,n-1,\vec \emptyset, (v_j-1)_1^{n-1}}\right)
         \\[10pt]
         -&
         M\left(H_{a+1,b,c-1,n,\vec \emptyset, (v_j-1)_1^n}\right)
         M\left(H_{a,b+1,c,n-1,\vec \emptyset, (v_j)_1^{n-1}}\right)
    \end{array}
\end{equation}

Applying the inductive hypothesis to each region except $H_{a,b,c,n,\vec \emptyset, (v_j)_1^n}$, this implies
\begin{equation}\label{eq:kuobase}
    \begin{array}{rl}
         &M\left(H_{a,b,c,n,\vec \emptyset, (v_j)_1^n}\right)
         \displaystyle\frac{\Pl(a+1,b+n,c-1)}{\prod_{j=1}^{n-1}(a+v_j)_{\underline{v_j}}}\\[15pt]
         \aequiv&
         \displaystyle\frac{\Pl(a,b+n+1,c-1)}{\prod_{j=1}^{n}(a+v_j)_{\underline{v_j}-1}}
         \cdot
         \displaystyle\frac{\Pl(a+1,b+n-1,c)}{\prod_{j=1}^{n-1}(a+v_j)_{\underline{v_j}+1}}\\[15pt]
         -&
         \displaystyle\frac{\Pl(a+1,b+n,c-1)}{\prod_{j=1}^{n}(a+v_j)_{\underline{v_j}}}
         \cdot
         \displaystyle\frac{\Pl(a,b+n,c)}{\prod_{j=1}^{n-1}(a+v_j)_{\underline{v_j}}}
    \end{array}
\end{equation}
We shall show that when the term $M(H_{a,b,c,n, \vec \emptyset, (v_j)_1^n})$ is replaced with $f_{b,c, \vec \emptyset, (v_j)_1^n}(a)$, the products on each line of \eqref{eq:kuobase} are $\aequiv$-equivalent, so that $$M(H_{a,b,c,n, \vec \emptyset, (v_j)_1^n}) \aequiv f_{b,c, \vec \emptyset, (v_j)_1^n}(a)$$ by relations $\eqref{eq:sum}$ and $\eqref{eq:substitute}$. It is clear that the product on the first line would be equal to the product on the third line when $f_{b,c,\vec \emptyset, (v_j)_1^n}(a)$ is written out explicitly, so it remains to show the products on the last two lines are $\aequiv$-equivalent. We will manipulate the terms on the second line to show this.
We rewrite
\begingroup
\allowdisplaybreaks
\begin{align}
\prod_{j=1}^n (a+v_j)_{\underline{v_j}-1} 
&= \prod_{j=1}^n (a+v_j)_{\underline{v_j}}/(a+c)_n\label{eq:sub1}\\
\prod_{j=1}^{n-1} (a+v_j)_{\underline{v_j}+1}
&=(a+c+1)_{n-1}\prod_{j=1}^{n-1} (a+v_j)_{\underline{v_j}}. \label{eq:sub2}
\end{align}
\endgroup

We can therefore rewrite:
\begin{align*}
        &\displaystyle\frac{\Pl(a,b+n+1,c-1)}{\prod_{j=1}^{n}(a+v_j)_{\underline{v_j}-1}}
        \cdot
        \displaystyle\frac{\Pl(a+1,b+n-1,c)}{\prod_{j=1}^{n-1}(a+v_j)_{\underline{v_j}+1}}\\[15pt]
        \hspace{-.5cm}(\mbox{by eqns. \eqref{eq:sub1}, \eqref{eq:sub2}})
        = & 
        \displaystyle\frac{\Pl(a,b+n+1,c-1)}{\prod_{j=1}^n (a+v_j)_{\underline{v_j}}}
        \cdot
        \displaystyle\frac{\Pl(a+1,b+n-1,c)}{\prod_{j=1}^{n-1} (a+v_j)_{\underline{v_j}} }
        \cdot
        \frac{(a+c)_n}{(a+c+1)_{n-1}}\\[15pt]
(\mbox{by eqn. \eqref{eq:sub}})\aequiv&
        \displaystyle\frac{\Pl(a,b+n,c)}{\prod_{j=1}^n (a+v_j)_{\underline{v_j}}}
        \cdot \frac{(a+b+n+1)_{c-1}}{(a+c)_{b+n}}
        \\[15pt]
        \times &
        \displaystyle\frac{\Pl(a+1,b+n,c-1)}{\prod_{j=1}^{n-1} (a+v_j)_{\underline{v_j}}}
        \cdot \frac{(a+1+c)_{b+n-1}}{(a+b+n+1)_{c-1}}
        \\[15pt]
        \times & \frac{(a+c)_n}{(a+c+1)_{n-1}}\\[15pt]
        = & 
        \displaystyle\frac{\Pl(a+1,b+n,c-1)}{\prod_{j=1}^{n}(a+v_j)_{\underline{v_j}}}
        \cdot
         \displaystyle\frac{\Pl(a,b+n+1,c)}{\prod_{j=1}^{n-1}(a+v_j)_{\underline{v_j}}},
\end{align*}
which is exactly the product on the third line of equation \eqref{eq:kuobase}. This completes the proof.
\end{proof}

Note that $H_{0,b,c,n,\vec \emptyset, (v_j)_1^n}$ has forced lozenges at its northern tip that when removed leave a region of the form $H_{1,b,c+1-v_1,n-1,\vec \emptyset, (v_j-v_1)_2^n}$, which is also a dented hexagon with fewer dents which are still all along the northwest side. Through repeated application of \thref{CLP}, one can obtain a complete product formula for $\M(H_{a,b,c,n,\vec \emptyset, \vec v})$; in particular, when $b=c=0$, this may be regarded as an independent proof for the number of tilings of a trapezoid with an arbitrary number of dents along the long base, which Cohn, Larsen, and Propp \cite{CLP} attribute to Gelfand and Tsetlin \cite{GT}. Our main theorem can be interpreted as a generalization of that result.

We are now ready to prove \thref{main}. The proof is similar to that of \thref{CLP}, and uses that lemma as a base case.

\begin{proof}[Proof of \thref{main}] 
We shall show by induction that if $H_{a,b,c,t,\vec u, \vec v}$ is tileable then
\begin{equation}\label{eq:eqmain}
\brac{\M(H_{a,b,c,t,\vec u, \vec v})} \aequiv \brac{ f_{b,c, \vec u, \vec v}(a)}:=\frac{\Pl(a,b+n, c+m)}{\prod_{i=1}^m(a+u_i)_{\underline{u_i}}\prod_{j=1}^n(a+v_j)_{\underline{v_j}}}.
\end{equation}
We will induct on the number of dents $m+n$, using $m=0$ and $n=0$ as base cases. In these cases, equation \eqref{eq:eqmain} follows immediately from \thref{CLP}.

For our inductive hypothesis, suppose equation \eqref{eq:eqmain} holds for dented hexagons with fewer than $m+n$ dents. Consider a dented hexagon with dents indexed by $(u_i)_1^m, (v_j)_1^n$. 

In the case where $\underline{u_m}=0$, $H_{a,b,c,t,(u_i)_1^m,(v_j)_1^n}$ has forced lozenges along its southeast side and therefore has the same tiling function as $H_{a,b,c+1,t-1,(u_i)_1^{m-1},(v_j)_1^n}$, so
\begin{align*}
    \brac{\M(H_{a,b,c,t,(u_i)_1^m,(v_j)_1^n})} &\aequiv \brac{\M(H_{a,b,c+1,t-1,(u_i)_1^{m-1},(v_j)_1^n})}\\
    (\mbox{IH}) & \aequiv \brac{f_{b,c+1,(u_i)_1^{m-1},(v_j)_1^n}(a)}\\
    (\mbox{\thref{loptop}c}) & \aequiv \brac{f_{b,c,(u_i)_1^{m},(v_j)_1^n}(a)}.
\end{align*}
Equation \eqref{eq:eqmain} holds when $\underline{v_n}=0$ by symmetric reasoning.

Assume therefore that $m,n>0$, $\underline{u_m}>0$ and $\underline{v_n}>0$.
Regard $H_{a,b,c,t,(u_i)_1^m,(v_j)_1^n}$ as a subregion of the (unbalanced) region $R_a:= H_{a,b,c,t,(u_i)_{1}^{m-1}, (v_j)_{1}^{n-1}}$. Within $R_a$ let $\alpha$ denote the unit triangle indexed by $v_n$, let $\beta$ denote the triangle indexed by $u_m$, let $\gamma$ denote the southmost  triangle along the northeast side of $R_a$, and let $\delta$ denote the southmost  triangle along the northwest side of $R_a$. Since $m,n>0$ and $\underline{u_m}>0$ and $\underline{v_n}>0$ these locations are clearly defined within $R_a$ and are distinct, as depicted in Figure \ref{fig:fig5}.

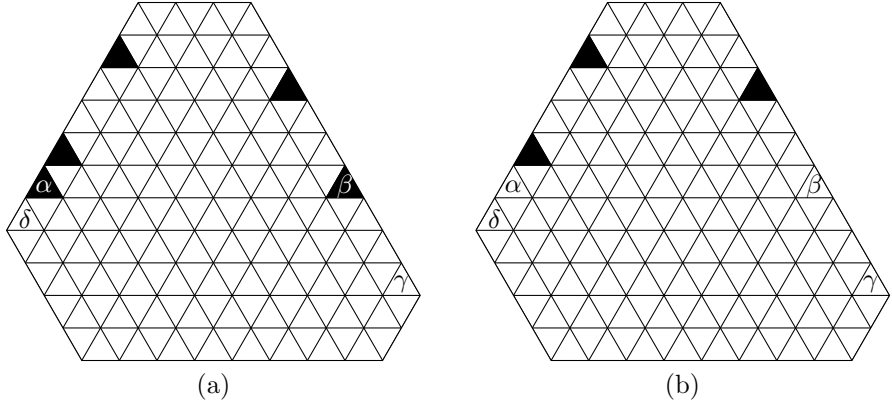
\begin{figure}
\begin{minipage}[c]{\textwidth}
\begin{center}
\begin{multicols}{2}

\begin{tikzpicture}[x={(0:1cm)},y={(120:1cm)},z={(240:1cm)},scale=.5]
    
    \hex{3}{4}{2}{5}
    \dentB{3}{3}
    \dentB{3}{6}
    \dentC{3}{2}
    \dentC{3}{5}
    \dentC{3}{6}
    
    \node (A) at (0,-.33,8.33){\color{white}$\alpha$};
    \node (B) at (0,-8.33,.33){\color{white}$\beta$};
    \node (C) at (0,-11.33,.33){$\gamma$};
    \node (D) at (0,-.33,9.33){$\delta$};
    
\end{tikzpicture}

(a)

\vfill \null \columnbreak

\begin{tikzpicture}[x={(0:1cm)},y={(120:1cm)},z={(240:1cm)},scale=.5]
    
    \hex{3}{4}{2}{5}
    \dentB{3}{3}
    \dentC{3}{2}
    \dentC{3}{5}
    
    \node (A) at (0,-.33,8.33){$\alpha$};
    \node (B) at (0,-8.33,.33){$\beta$};
    \node (C) at (0,-11.33,.33){$\gamma$};
    \node (D) at (0,-.33,9.33){$\delta$};
    
\end{tikzpicture}

(b)
\end{multicols}
\end{center}
\caption{(a) $H_{3,4,2,5,(3,6),(2,5,6)}$; (b) $R_3$ with $\alpha, \beta, \gamma, \delta$ labeled.} 
    \label{fig:fig5}
\end{minipage}
\end{figure}

We will apply Kuo Condensation (\thref{kuo}) to $R_a, \alpha, \beta, \gamma, \delta$. Figure \ref{fig:fig6} depicts each region referenced in the condensation formula, in some cases with forced lozenges shaded (we regard these as not being part of the region). Note they are all families of balanced dented hexagons.

Since $H_{a,b,c,t,\vec u, \vec v}$ is tileable so are each of these regions by \thref{existcor}.
The number of dents is strictly minimal on $R_a - \alpha - \beta$, so we may apply the inductive hypothesis to each of the other regions. Recall lemma 2 states
$$
    M(R-\alpha-\beta)M(R-\beta-\delta) = M(R-\alpha-\delta)M(R-\beta-\gamma) - M(R-\alpha-\gamma)M(R-\beta-\delta).
$$

Removing forced lozenges as depicted in Figure \ref{fig:fig6}, we can rewrite this:
\begin{equation}
    \begin{array}{lr}
         & M\left(H_{a,b,c,t,(u_i)_1^{m},(v_j)_1^{n}}\right)
         M\left(H_{a,b+1,c+1,t-2,(u_i)_1^{m-1},(v_j)_1^{n-1}}\right)\\[10 pt]
         =&
         M\left(H_{a,b+1,c,t-1,(u_i)_1^{m-1},(v_j)_1^{n}}\right)
         M\left(H_{a,b,c+1,t-1,(u_i)_1^{m},(v_j)_1^{n-1}}\right)\\[10 pt]
         -&
         M\left(H_{a,b,c+1,t-1,(u_i)_1^{m-1},(v_j)_1^{n}}\right)
         M\left(H_{a,b+1,c,t-1,(u_i)_1^{m},(v_j)_1^{n-1}}\right)
    \end{array}
\end{equation}
Applying the inductive hypothesis to each region except $H_{a,b,c,t,(u_i)_1^m,(v_j)_1^n}$, this implies
\begin{equation}\label{eq:kuomain}
    \begin{array}{rl}
         &M\left(H_{a,b,c,t,(u_i)_1^{m},(v_j)_1^{n}}\right)
         \displaystyle\frac{\Pl(a,b+n,c+m)}{\prod_{i=1}^{m-1} (a+u_i)_{\underline{u_i}} \prod_{j=1}^{n-1} (a+v_j)_{\underline{v_j}}}\\[15pt]
         \aequiv& 
         \displaystyle\frac{\Pl(a,b+n+1,c+m-1)}{\prod_{i=1}^{m-1} (a+u_i)_{\underline{u_i}+1} \prod_{j=1}^{n} (a+v_j)_{\underline{v_j}-1}}
         \displaystyle\frac{\Pl(a,b+n-1,c+m+1)}{\prod_{i=1}^{m} (a+u_i)_{\underline{u_i}-1} \prod_{j=1}^{n-1} (a+v_j)_{\underline{v_j}+1}}\\[15pt]
         -&
         \displaystyle \frac{\Pl(a,b+n,c+m)}{\prod_{i=1}^{m-1} (a+u_i)_{\underline{u_i}} \prod_{j=1}^{n} (a+v_j)_{\underline{v_j}}}
         \displaystyle\frac{\Pl(a,b+n,c+m)}{\prod_{i=1}^{m} (a+u_i)_{\underline{u_i}} \prod_{j=1}^{n-1} (a+v_j)_{\underline{v_j}}}
    \end{array}
\end{equation}

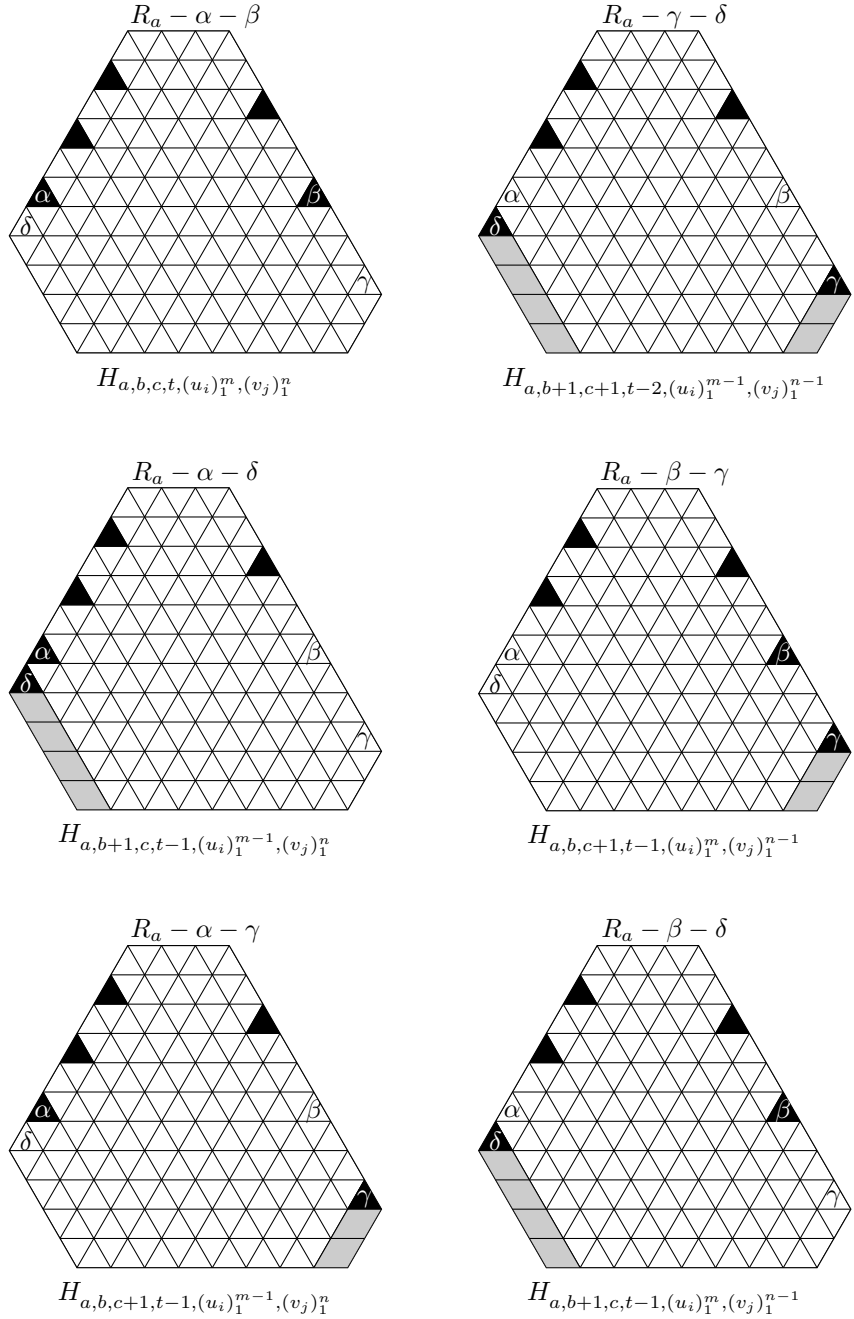
\begin{figure}
\begin{minipage}[c]{\textwidth}
\begin{center}
\begin{multicols}{2}

$R_a - \alpha - \beta$

\begin{tikzpicture}[x={(0:1cm)},y={(120:1cm)},z={(240:1cm)},scale=.45]
    
    \dentC{3}{6}
    \node (A) at (0,-.33,8.33){\color{white}$\alpha$};
    
    \dentB{3}{6}
    \node (B) at (0,-8.33,.33){\color{white}$\beta$};
    
    \node (C) at (0,-11.33,.33){$\gamma$};
    
    \node (D) at (0,-.33,9.33){$\delta$};
    
    \hex{3}{4}{2}{5}
    \dentB{3}{3}
    \dentC{3}{2}
    \dentC{3}{4}
    
\end{tikzpicture}

$H_{a,b,c,t,(u_i)_1^{m},(v_j)_1^{n}}$

\vfill \null \columnbreak

$R_a - \gamma - \delta$

\begin{tikzpicture}[x={(0:1cm)},y={(120:1cm)},z={(240:1cm)},scale=.45]

    \hex{3}{4}{2}{5}
    
    \node (A) at (0,-.33,8.33){$\alpha$};
    
    \node (B) at (0,-8.33,.33){$\beta$};
    
    \draw[fill=black!20] (0,-12,0) -- ++(0,0,2) -- ++(-1,0,0) -- ++(0,0,-3) -- cycle;
    \dentB{3}{9}
    \node (C) at (0,-11.33,.33){\color{white}$\gamma$};
    
    \draw[fill=black!20] (0,0,10) -- ++(0,-4,0) -- ++(1,0,0) -- ++(0,5,0) -- cycle;
    \dentC{3}{7}
    \node (D) at (0,-.33,9.33){\color{white}$\delta$};
    
    \dentB{3}{3}
    \dentC{3}{2}
    \dentC{3}{4}

    \draw (0,-1,10) -- ++(1,0,0);
    \draw (0,-2,10) -- ++(1,0,0);
    \draw (0,-3,10) -- ++(1,0,0);
    
    \draw (0,-12,1) -- ++(-1,0,0);
    
\end{tikzpicture}

$H_{a,b+1,c+1,t-2,(u_i)_1^{m-1},(v_j)_1^{n-1}}$

\end{multicols}
\begin{multicols}{2}

$R_a - \alpha - \delta$

\begin{tikzpicture}[x={(0:1cm)},y={(120:1cm)},z={(240:1cm)},scale=.45]
    \hex{3}{4}{2}{5}
    
    \dentC{3}{6}
    \node (A) at (0,-.33,8.33){\color{white}$\alpha$};
    
    \node (B) at (0,-8.33,.33){$\beta$};
    
    \node (C) at (0,-11.33,.33){$\gamma$};
    
    \draw[fill=black!20] (0,0,10) -- ++(0,-4,0) -- ++(1,0,0) -- ++(0,5,0) -- cycle;
    \dentC{3}{7}
    \node (D) at (0,-.33,9.33){\color{white}$\delta$};
    
    \dentB{3}{3}
    \dentC{3}{2}
    \dentC{3}{4}
    
    \draw (0,-1,10) -- ++(1,0,0);
    \draw (0,-2,10) -- ++(1,0,0);
    \draw (0,-3,10) -- ++(1,0,0);
    
\end{tikzpicture}

$H_{a,b+1,c,t-1,(u_i)_1^{m-1},(v_j)_1^{n}}$

\vfill \null \columnbreak

$R_a - \beta - \gamma$

\begin{tikzpicture}[x={(0:1cm)},y={(120:1cm)},z={(240:1cm)},scale=.45]

    \hex{3}{4}{2}{5}
    
    \node (A) at (0,-.33,8.33){$\alpha$};
    
    \dentB{3}{6}
    \node (B) at (0,-8.33,.33){\color{white}$\beta$};
    
    \draw[fill=black!20] (0,-12,0) -- ++(0,0,2) -- ++(-1,0,0) -- ++(0,0,-3) -- cycle;
    \dentB{3}{9}
    \node (C) at (0,-11.33,.33){\color{white}$\gamma$};
    
    \node (D) at (0,-.33,9.33){$\delta$};
    
    \dentB{3}{3}
    \dentC{3}{2}
    \dentC{3}{4}
    
    \draw (0,-12,1) -- ++(-1,0,0);
    
\end{tikzpicture}

$H_{a,b,c+1,t-1,(u_i)_1^{m},(v_j)_1^{n-1}}$

\end{multicols}
\begin{multicols}{2}

$R_a - \alpha - \gamma$

\begin{tikzpicture}[x={(0:1cm)},y={(120:1cm)},z={(240:1cm)},scale=.45]

    \hex{3}{4}{2}{5}
    
    \dentC{3}{6}
    \node (A) at (0,-.33,8.33){\color{white}$\alpha$};
    
    \node (B) at (0,-8.33,.33){$\beta$};
    
    \draw[fill=black!20] (0,-12,0) -- ++(0,0,2) -- ++(-1,0,0) -- ++(0,0,-3) -- cycle;
    \dentB{3}{9}
    \node (C) at (0,-11.33,.33){\color{white}$\gamma$};
    
    \node (D) at (0,-.33,9.33){$\delta$};
    \dentB{3}{3}
    \dentC{3}{2}
    \dentC{3}{4}
    
    \draw (0,-12,1) -- ++(-1,0,0);
    
\end{tikzpicture}

$H_{a,b,c+1,t-1,(u_i)_1^{m-1},(v_j)_1^{n}}$

\vfill \null \columnbreak

$R_a - \beta - \delta$

\begin{tikzpicture}[x={(0:1cm)},y={(120:1cm)},z={(240:1cm)},scale=.45]

    \hex{3}{4}{2}{5}
    
    \node (A) at (0,-.33,8.33){$\alpha$};
    
    \dentB{3}{6}
    \node (B) at (0,-8.33,.33){\color{white}$\beta$};
    
    \node (C) at (0,-11.33,.33){$\gamma$};
    
    \draw[fill=black!20] (0,0,10) -- ++(0,-4,0) -- ++(1,0,0) -- ++(0,5,0) -- cycle;
    \dentC{3}{7}
    \node (D) at (0,-.33,9.33){\color{white}$\delta$};
    
    \dentB{3}{3}
    \dentC{3}{2}
    \dentC{3}{4}
    
    \draw (0,-1,10) -- ++(1,0,0);
    \draw (0,-2,10) -- ++(1,0,0);
    \draw (0,-3,10) -- ++(1,0,0);
    
\end{tikzpicture}

$H_{a,b+1,c,t-1,(u_i)_1^{m},(v_j)_1^{n-1}}$

\end{multicols}
\end{center}
\caption{$R_a$ with two of $\alpha, \beta, \gamma, \delta$ removed and forced lozenges shaded.} 
    \label{fig:fig6}
\end{minipage}
\end{figure}

We shall show that when the term $M\left(H_{a,b,c,t,(u_i)_1^{m},(v_j)_1^{n}}\right)$ is replaced with $f_{b,c,(u_i)_1^m,(v_j)_1^n}(a)$, the products on each line of \eqref{eq:kuomain} are $\aequiv$-equivalent, so that
$$M(H_{a,b,c,n, (u_i)_1^m, (v_j)_1^n}) \aequiv f_{b,c, (u_i)_1^m, (v_j)_1^n}(a)$$ by relations $\eqref{eq:sum}$ and $\eqref{eq:substitute}$. It is clear that the product on the first line would be equal to the product on the third line when $f_{b,c,\vec \emptyset, (v_j)_1^n}(a)$ is written out explicitly, so it remains to show the products on the last two lines are $\aequiv$-equivalent. We will manipulate the terms on the second line to show this, employing the following analogues to equations \eqref{eq:sub1} and \eqref{eq:sub2}:
\begingroup
\allowdisplaybreaks
\begin{align}
    \prod_{j=1}^k (a+v_j)_{\underline{v_j}-1} &= \prod_{j=1}^k(a+v_j)_{\underline{v_j}} / (a+c+m)_k \label{eq:sub3}\\
    \prod_{i=1}^k (a+u_i)_{\underline{u_i}-1} &= \prod_{i=1}^k(a+u_j)_{\underline{u_i}} / (a+b+n)_k\\
    \prod_{j=1}^k (a+v_j)_{\underline{v_j}+1} &= (a+c+m+1)_k\prod_{j=1}^k(a+v_j)_{\underline{v_j}} \\
    \prod_{i=1}^k (a+u_i)_{\underline{u_i}-1} &= (a+b+n+1)_k\prod_{i=1}^k(a+u_j)_{\underline{u_i}}. \label{eq:sub4}
\end{align}
\endgroup

We can therefore rewrite:
\begin{align*}
    &\displaystyle\frac{\Pl(a,b+n+1,c+m-1)}{\prod_{i=1}^{m-1} (a+u_i)_{\underline{u_i}+1} \prod_{j=1}^{n} (a+v_j)_{\underline{v_j}-1}}
    \cdot
    \frac{\Pl(a,b+n-1,c+m+1)}{\prod_{i=1}^{m} (a+u_i)_{\underline{u_i}-1} \prod_{j=1}^{n-1} (a+v_j)_{\underline{v_j}+1}}\\[10pt]
    \hspace{-1.5cm}(\mbox{by eqns. \eqref{eq:sub3} - \eqref{eq:sub4}}) 
    = & 
    \frac{\Pl(a,b+n+1,c+m-1)}{\prod_{i=1}^{m-1} (a+u_i)_{\underline{u_i}} \prod_{j=1}^{n} (a+v_j)_{\underline{v_j}}} \cdot \frac{(a+c+m)_{n}}{(a+b+n+1)_{m-1}}\\[10pt]
    \times &
    \frac{\Pl(a,b+n-1,c+m+1)}{\prod_{i=1}^{m} (a+u_i)_{\underline{u_i}} \prod_{j=1}^{n-1} (a+v_j)_{\underline{v_j}}}\cdot
    \frac{(a+b+n)_{m}}{(a+c+m+1)_{n-1}}\\[10pt]
    \hspace{-1.5cm}(\mbox{by \thref{sub}})
    \aequiv &
    \frac{\Pl(a,b+n,c+m)}{\prod_{i=1}^{m-1} (a+u_i)_{\underline{u_i}} \prod_{j=1}^{n} (a+v_j)_{\underline{v_j}}} 
    \cdot \frac{(a+b+n+1)_{c+m-1}}{(a+c+m)_{b+n}}
    \cdot \frac{(a+c+m)_{n}}{(a+b+n+1)_{m-1}}\\[10pt]
    \times &
    \frac{\Pl(a,b+n,c+m)}{\prod_{i=1}^{m} (a+u_i)_{\underline{u_i}} \prod_{j=1}^{n-1} (a+v_j)_{\underline{v_j}}}
    \cdot \frac{(a+c+m+1)_{b+n-1}}{(a+b+n)_{c+m}}
    \cdot \frac{(a+b+n)_{m}}{(a+c+m+1)_{n-1}}\\[10pt]
    =& \displaystyle \frac{\Pl(a,b+n,c+m)}{\prod_{i=1}^{m-1} (a+u_i)_{\underline{u_i}} \prod_{j=1}^{n} (a+v_j)_{\underline{v_j}}}
    \cdot
    \displaystyle\frac{\Pl(a,b+n,c+m)}{\prod_{i=1}^{m} (a+u_i)_{\underline{u_i}} \prod_{j=1}^{n-1} (a+v_j)_{\underline{v_j}}},
\end{align*}
    which is exactly the product from third line from \eqref{eq:kuomain}.
    This is the last case in the inductive step, so equation \eqref{eq:eqmain} holds in general.
\end{proof}

\section{Hexagons with Two Large Dents}  \label{corollaries}

When the dents along each side of a dented hexagon are all adjacent, as in $H_{a,b,c,t,(u+i)_{i=1}^{m}, (v+j)_{j=1}^{n}}$, the region has forced lozenges that form large triangular dents. Figure \ref{fig:fig3} depicts this forcing, and the region that results if the forced lozenges and omitted dents are removed entirely.

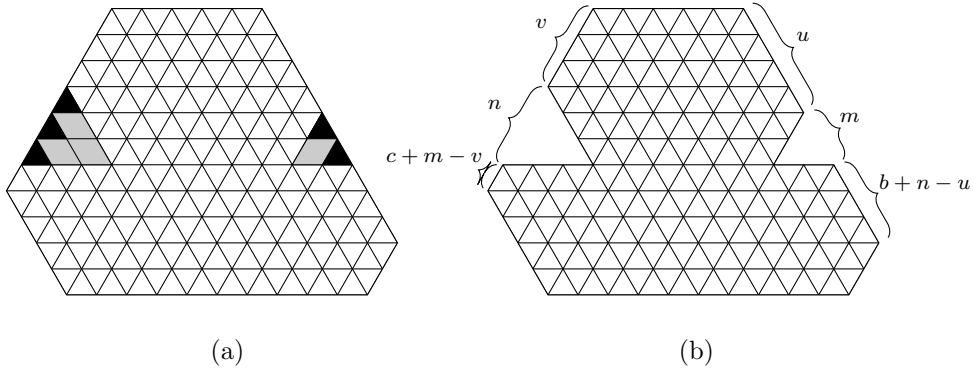
\begin{figure}
\begin{center}

\begin{tikzpicture}[x={(0:1cm)},y={(120:1cm)},z={(240:1cm)},scale=.4]

    \begin{scope}
    \hex{5}{4}{2}{5}
    \draw[fill=black!20] (0,0,8) -- ++(0,-3,0) -- ++(-3,0,0) -- cycle;
    
    \draw[fill=black!20] (0,-9,0) -- ++(0,0,2) --++(2,0,0) -- cycle;
    \dentB{5}{5}
    \dentB{5}{6}
    \dentC{5}{4}
    \dentC{5}{5}
    \dentC{5}{6}
    
    \draw (0,0,10) --++(2,0,0);
    \draw (0,0,9) --++(0,-2,0);
    \end{scope}
    
    \begin{scope}[xshift=16cm]

    \draw [decorate,decoration={brace,amplitude=5pt,raise=4pt}]
    (0,-5,0) -- ++(0,-3.9,0) node [black,midway,yshift=.3cm,xshift=.4cm] {\footnotesize
    $u$};
    
    \draw [decorate,decoration={brace,amplitude=5pt,raise=4pt}]
    (0,-9.1,0) -- ++(0,-1.9,0) node [black,midway,yshift=.3cm,xshift=.4cm] {\footnotesize
    $m$};
    
    \draw [decorate,decoration={brace,amplitude=5pt,raise=4pt}]
    (0,-11.1,0) -- ++(0,-2.9,0) node [black,midway,yshift=.3cm,xshift=.9cm] {\footnotesize
    $b+n-u$};
    
    \draw [decorate,decoration={brace,amplitude=5pt,raise=2pt}]
    (0,0,7.9) -- ++(0,0,-2.9) node [black,midway,yshift=.3cm,xshift=-.4cm] {\footnotesize
    $v$};
    
    \draw [decorate,decoration={brace,amplitude=5pt,raise=2pt}]
    (0,0,11) -- ++(0,0,-2.9) node [black,midway,yshift=.3cm,xshift=-.4cm] {\footnotesize
    $n$};
    
    \draw [decorate,decoration={brace,amplitude=5pt,raise=2pt}]
    (0,0,12) -- ++(0,0,-.9) node [black,midway,yshift=.3cm,xshift=-.8cm] {\footnotesize
    $c+m-v$};

\hexTwoDent{5}{4}{2}{4}{2}{3}{3}
    \end{scope}
    
\end{tikzpicture}

\begin{multicols}{2}

(a)

(b)
\end{multicols}
\end{center}
\caption{(a) $H_{5,4,2,5,(4+i)_1^2,(3+j)_1^3}$; (b) the region with forced lozenges removed} 
    \label{fig:fig3}
\end{figure}

The original goal of this paper was finding a general tiling function for hexagons with two large dents, and indeed \thref{main} simplifies to a ratio of tilings of semiregular hexagons when the dents along each side are all adjacent:

\begin{corollary}\thlabel{twodents}
     A dented hexagon $H_{a,b,c,t,(u+i)_{i=1}^{m}, (v+j)_{j=1}^{n}}$ has tilings exactly if $u\geq n$ or $v\geq m$, in which case
    \begin{align*}\hspace{-.5cm}\M\left(H_{a,b,c,t,(u+i)_{i=1}^{m}, (v+j)_{j=1}^{n}}\right) =& \M\left(H_{0,b,c,t,(u+i)_{i=1}^{m}, (v+j)_{j=1}^{n}}\right)\\
    \times & \frac{\Pl(a,b+n,c+m)\Pl(u,b+n-u,m)\Pl(v,c+m-v,n)}{\Pl(a+u,b+n-u,m)\Pl(a+v,c+m-v,n)}.
    \end{align*}
\end{corollary}

This generalizes a specific case of a result by Lai who studied the problem when the southern borders of the dents are level (see Theorem 3.1 from \cite{La17}, with $q=1$).
We give an expression of that result below in the language of this paper; it follows from \thref{twodents}.

\begin{corollary}\thlabel{TLcorr}
    Given a dented hexagon $H_{a,b,c,t,(u+i)_{i=1}^{m}, (v+j)_{j=1}^{n}}$ with $u+m=v+n$, let $D:=u-n$. If $D<0$ the region has no tilings. Otherwise,
    \begin{align*}
        \M\left(H_{a,b,c,t,(u+i)_{i=1}^{m}, (v+j)_{j=1}^{n}}\right) =& \frac{\Pl(a,b+n,c+m)\Pl(u,b-D,m)\Pl(v,c-D,n)}{\Pl(a+u,b-D,m)\Pl(a+v,c-D,n)}\\
        \times & \frac{\Pl(c-D,n+m,b)\Pl(D,n,m)}{\Pl(c-D+n,m,D)} \cdot \Pl(D,m,b-D).
    \end{align*}
\end{corollary}


\begin{proof}[Proof of \thref{twodents}]
It is straightforward to check for arbitrary values that 
\begin{equation}\label{eq:eqtwodent}
\begin{array}{rl}
    m > v, n > u \iff &  v < u+ m \leq v + n, \mbox{ and} \quad m+ (u+ m-v) > u+m \\
    \mbox{ OR}&  u < v+ n \leq u+ m,\mbox{ and}  \quad n + (v+ n-u) > v+n.
\end{array}
\end{equation}
We shall show the second set of expressions hold exactly when the region $H_{a,b,c,t,(u+i)_1^m,(v+j)_1^n}$ has no tilings.

Suppose that $u+m \leq v+n$ (meaning the southern border of the eastern dent is weakly north of the southern border of the western dent) and the region has no tilings.
Then $(\mu_N - N)$ is maximized at $N=u+m$: so $\mu_{u+m}>u+m$. 
If $v>u+m$ then $\mu_{u+m}=m\leq u+m$, giving a contradiction. So it must be that $v\leq u+m$. This implies $\mu_{u+m}=m+(u+m-v) > u+m$, so that the first line on the left side of equation \eqref{eq:eqtwodent} holds. The argument works in reverse: if $v<u+m\leq v+n$ and $(u+m)-v>u$ then $\mu_{u+m} > u+m$; so these three inequalities are equivalent to the eastern dent being weakly north of the western dent and the region having no tilings.

The inequalities $u < v+ n \leq u+ m$ and $n + (v+ n-u) > v+n$ are equivalent to the western dent being weakly north of the eastern dent and the region having no tilings.

The formula given follows immediately from \thref{main}.
\end{proof}

\begin{proof}[Proof of \thref{TLcorr}]
    The region $H_{0,b,c,t,(u+i)_{i=1}^{m}, (v+j)_{j=1}^{n}}$ has forced lozenges that when removed give a region congruent to $H_{c-D,n+m,b-D,D,(n+i)_{i=1}^D, \vec \emptyset}$, as depicted in figure \ref{fig:fig12}.
    
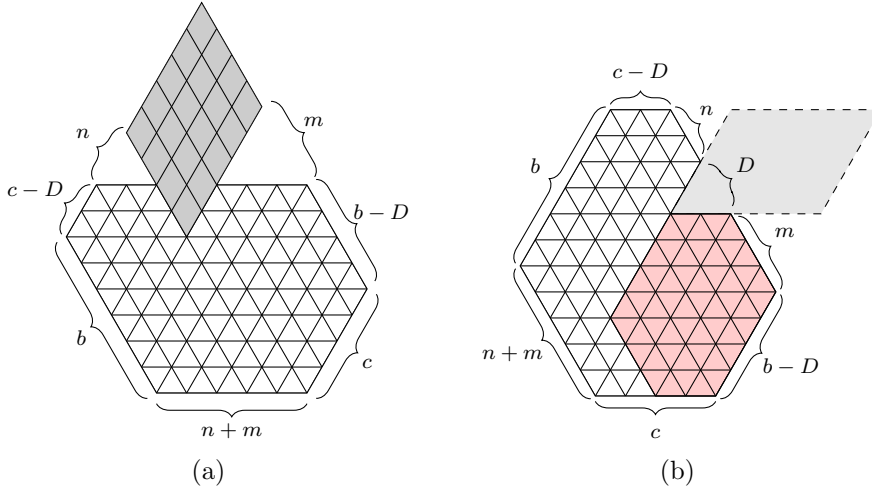
\begin{figure}
\begin{center}
\begin{multicols}{2}

\begin{tikzpicture}[x={(0:1cm)},y={(120:1cm)},z={(240:1cm)},scale=.4]
    
    \draw [decorate,decoration={brace,amplitude=5pt,raise=2pt}]
    (0,0,6.9) -- ++(0,0,-2) node [black,midway,yshift=.3cm,xshift=-.4cm] {\footnotesize
    $n$};
    
    \draw [decorate,decoration={brace,amplitude=5pt,raise=2pt}]
    (0,0,9) -- ++(0,0,-1.9) node [black,midway,yshift=.3cm,xshift=-.6cm] {\footnotesize
    $c-D$};
    
    \draw [decorate,decoration={brace,amplitude=5pt,raise=4pt}]
    (0,-4,0) -- ++(0,-2.9,0) node [black,midway,yshift=.3cm,xshift=.4cm] {\footnotesize
    $m$};
    
    \draw [decorate,decoration={brace,amplitude=5pt,raise=4pt}]
    (0,-7,0) -- ++(0,-3.9,0) node [black,midway,yshift=.3cm,xshift=.6cm] {\footnotesize
    $b-D$};
    
    \draw [decorate,decoration={brace,amplitude=5pt,raise=4pt}]
    (0,-11,0) -- ++(0,0,4) node [black,midway,yshift=-.3cm,xshift=.4cm] {\footnotesize
    $c$};
    
    \draw [decorate,decoration={brace,amplitude=5pt,raise=4pt}]
    (0,-11,4) -- ++(-5,0,0) node [black,midway,yshift=-.5cm] {\footnotesize
    $n+m$};
    
    \draw [decorate,decoration={brace,amplitude=5pt,raise=4pt}]
    (-5,-11,4) -- ++(0,6,0) node [black,midway,yshift=-.3cm,xshift=-.4cm] {\footnotesize
    $b$};

\hexTwoDent{0}{6}{4}{4}{3}{5}{2}

\draw[fill=black!20] (0,0,0) -- ++(0,0,5) -- ++(0,-4,0,) -- ++(0,0,-5) -- cycle;

\draw (0,0,1) -- ++(0,-4,0);
\draw (0,0,2) -- ++(0,-4,0);
\draw (0,0,3) -- ++(0,-4,0);
\draw (0,0,4) -- ++(0,-4,0);

\draw (0,-1,0) -- ++(0,0,5);
\draw (0,-2,0) -- ++(0,0,5);
\draw (0,-3,0) -- ++(0,0,5);
    
\end{tikzpicture}

(a)

\vfill \null \columnbreak

\null
\vfill

\begin{tikzpicture}[x={(0:1cm)},y={(120:1cm)},z={(240:1cm)},scale=.4]

\draw[dashed, fill=black!10] (4,0,6) -- ++(5,0,0) -- ++(0,0,-4) -- ++(-5,0,0) -- cycle;

\draw[ fill=red!20] (4,0,6) -- ++(2,0,0) -- ++(0,-3,0) -- ++(0,0,4) -- ++(-2,0,0) -- ++(0,3,0) -- cycle;

    \draw [decorate,decoration={brace,amplitude=5pt,raise=2pt}]
    (0,0,2) -- ++(2,0,0) node [black,midway,yshift=.5cm] {\footnotesize
    $c-D$};
    
    \draw [decorate,decoration={brace,amplitude=5pt,raise=2pt}]
    (0,-2,0) -- ++(0,-1.8,0) node [black,midway,yshift=.3cm,xshift=.3cm] {\footnotesize
    $n$};
    
    \draw [decorate,decoration={brace,amplitude=5pt,raise=2pt}]
    (0,-4.2,0) -- ++(0,-1.6,0) node [black,midway,yshift=.3cm,xshift=.4cm] {\footnotesize
    $D$};
    
    \draw [decorate,decoration={brace,amplitude=5pt,raise=2pt}]
    (0,-6.2,0) -- ++(0,-2.8,0) node [black,midway,yshift=.3cm,xshift=.4cm] {\footnotesize
    $m$};
    
    \draw [decorate,decoration={brace,amplitude=5pt,raise=2pt}]
    (0,-9,0) -- ++(0,0,4) node [black,midway,yshift=-.3cm,xshift=.6cm] {\footnotesize
    $b-D$};
    
    \draw [decorate,decoration={brace,amplitude=5pt,raise=2pt}]
    (0,-9,4) -- ++(-4,0,0) node [black,midway,yshift=-.5cm,xshift=.0cm] {\footnotesize
    $c$};
    
    \draw [decorate,decoration={brace,amplitude=5pt,raise=2pt}]
    (-4,-9,4) -- ++(0,5,0) node [black,midway,yshift=-.3cm,xshift=-.6cm] {\footnotesize
    $n+m$};
    
    \draw [decorate,decoration={brace,amplitude=5pt,raise=2pt}]
    (0,0,8) -- ++(0,0,-6) node [black,midway,yshift=.3cm,xshift=-.4cm] {\footnotesize
    $b$};

\hexTwoDent{2}{5}{4}{2}{2}{2}{0}
    
\end{tikzpicture}

(b)

\null \columnbreak
\end{multicols}
\end{center}
\caption{(a) $H_{0,6,4,5,(4+i)_1^3, (5+j)_1^2}$ with forced lozenges; (b) the region interpreted as a hexagon with one dent ($H_{D,m,b-D}$ shaded red).
} 
    \label{fig:fig12}
\end{figure}

    The region $H_{0,n+m,b-D,D,(n+i)_{i=1}^D, \vec \emptyset}$ also has forced lozenges, that when removed give a semiregular hexagon $H_{D,m,b-D}$.
    The result then follows by applying \thref{twodents} to both $H_{a,b,c,t,(u+i)_{i=1}^{m}, (v+j)_{j=1}^{n}}$ and $H_{c-D,n+m,b-D,D,(n+i)_{i=1}^D, \vec \emptyset}$.
\end{proof}

\section{Final Remarks}

The method of proof used for \thref{TLcorr} is applicable whenever the tiling function of $H_{0,b,c,t,\vec u, \vec v}$ is simple to express. For example, let $H=H_{0,b,c,t,(u+i)_1^m, (v+j)_1^n}$ be a region with $\underline{v_n}=1$, as depicted in Figure \ref{fig:fig10} (a). 

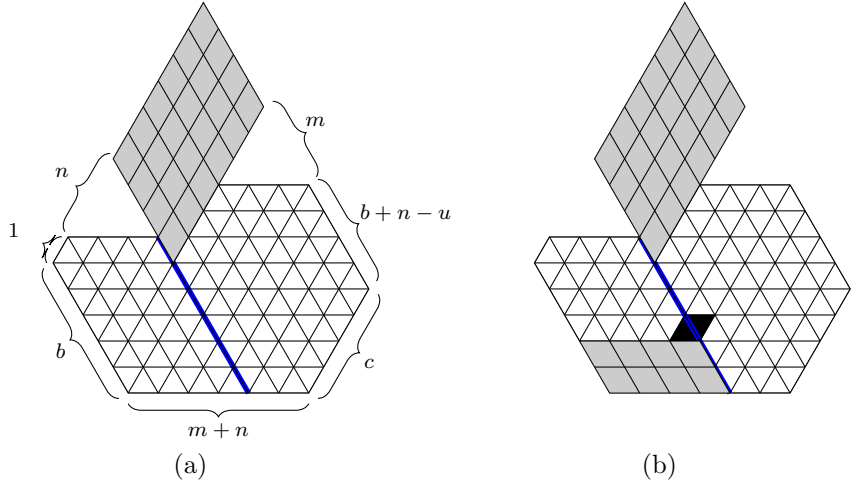
\begin{figure}
\begin{center}

\begin{tikzpicture}[x={(0:1cm)},y={(120:1cm)},z={(240:1cm)},scale=.4]

    \begin{scope}
    \draw[line width=2pt, blue](0,-3,6)-- ++(0,-6,0);
    
    \draw [decorate,decoration={brace,amplitude=5pt,raise=2pt}]
    (0,0,8.9) -- ++(0,0,-3) node [black,midway,yshift=.3cm,xshift=-.4cm] {\footnotesize
    $n$};
    
    \draw [decorate,decoration={brace,amplitude=5pt,raise=2pt}]
    (0,0,10.1) -- ++(0,0,-1) node [black,midway,yshift=.3cm,xshift=-.6cm] {\footnotesize
    $1$};
    
    \draw [decorate,decoration={brace,amplitude=5pt,raise=4pt}]
    (0,-4,0) -- ++(0,-2.9,0) node [black,midway,yshift=.3cm,xshift=.4cm] {\footnotesize
    $m$};
    
    \draw [decorate,decoration={brace,amplitude=5pt,raise=4pt}]
    (0,-7,0) -- ++(0,-3.9,0) node [black,midway,yshift=.3cm,xshift=.9cm] {\footnotesize
    $b+n-u$};
    
    \draw [decorate,decoration={brace,amplitude=5pt,raise=4pt}]
    (0,-11,0) -- ++(0,0,4) node [black,midway,yshift=-.3cm,xshift=.4cm] {\footnotesize
    $c$};
    
    \draw [decorate,decoration={brace,amplitude=5pt,raise=4pt}]
    (0,-11,4) -- ++(-6,0,0) node [black,midway,yshift=-.5cm] {\footnotesize
    $m+n$};
    
    \draw [decorate,decoration={brace,amplitude=5pt,raise=4pt}]
    (-6,-11,4) -- ++(0,5,0) node [black,midway,yshift=-.3cm,xshift=-.4cm] {\footnotesize
    $b$};

    \hexTwoDent{0}{5}{4}{4}{3}{6}{3}
    
    \draw[fill=black!20] (0,0,0) -- ++(0,0,6) -- ++(0,-4,0,) -- ++(0,0,-6) -- cycle;

    \draw (0,0,1) -- ++(0,-4,0);
    \draw (0,0,2) -- ++(0,-4,0);
    \draw (0,0,3) -- ++(0,-4,0);
    \draw (0,0,4) -- ++(0,-4,0);
    \draw (0,0,5) -- ++(0,-4,0);
    
    \draw (0,-1,0) -- ++(0,0,6);
    \draw (0,-2,0) -- ++(0,0,6);
    \draw (0,-3,0) -- ++(0,0,6);
    \end{scope}
    
    \begin{scope}[xshift=16cm]
    
    \draw[fill=black] (0,-7,6) -- ++(0,0,-1) -- ++(-1,0,0) -- ++(0,0,1) -- cycle;
    
    \draw[line width=2pt, blue](0,-3,6)-- ++(0,-6,0);

    \hexTwoDent{0}{5}{4}{4}{3}{6}{3}
    
    \draw[fill=black!20] (0,0,0) -- ++(0,0,6) -- ++(0,-4,0,) -- ++(0,0,-6) -- cycle;

    \draw (0,0,1) -- ++(0,-4,0);
    \draw (0,0,2) -- ++(0,-4,0);
    \draw (0,0,3) -- ++(0,-4,0);
    \draw (0,0,4) -- ++(0,-4,0);
    \draw (0,0,5) -- ++(0,-4,0);
    
    \draw (0,-1,0) -- ++(0,0,6);
    \draw (0,-2,0) -- ++(0,0,6);
    \draw (0,-3,0) -- ++(0,0,6);
    
    \draw[fill=black!20] (0,-7,6) -- ++(0,-2,0) -- ++(-4,0,0) -- ++(0,2,0) -- cycle;
    
    \draw (1,-3,10) --++(0,-2,0);
    \draw (2,-3,10) --++(0,-2,0);
    \draw (3,-3,10) --++(0,-2,0);
    
    \draw (0,-4,10) --++(4,0,0);
    \end{scope}
    
\end{tikzpicture}

\vspace{-.5cm}

\begin{multicols}{2}
(a)

(b)
\end{multicols}
\end{center}
\caption{(a) $H_{0,b,c, \vec u, \vec v}$ with $\underline{v_n}=1$ and a split-line; (b) $R_3$ with forced lozenges shaded grey.
} 
    \label{fig:fig10}
\end{figure}

Observe that the blue split-line in the figure partitions the region into two unbalanced hexagons $H_{n,b,0,1}$ and $H_{m-1,b+n-u,c-1,1}$.
It can be seen by modifying the proof of the Region Splitting Lemma that if $R$ is a balanced region with a partition into regions $P$, $Q$, so that unit triangles in $P$ which are adjacent to unit triangles in $Q$ are all of the same orientation, and this orientation \emph{is in excess} within $P$ by some amount, say $d$, then all tilings of $R$ include exactly $d$ lozenges covering one  triangle from $P$ and one  triangle from $Q$.

It follows that all tilings of the entire region must include exactly one lozenge which crosses the split-line.
Let $S_i$ be the tilings of the region so that the split-line-crossing lozenge's northern border is $i$ units north of $H$'s southern side. Let $R_i$ be the region obtained by removing that lozenge from $H$, and observe that $S_i$ has a natural bijection with the tilings of $R_i$. Furthermore, \thref{splitting} applies to each region $R_i$ with respect to the blue split-line, and partitions $R_i$ into two regions with known tiling functions:
\begingroup
\allowdisplaybreaks
\begin{align*}
    \M(R_i) &= \M(H_{1,n,b+1-i})\M(H_{m-1,b+n-u,c-1,1,(i),\vec \emptyset})\\[5pt]
    &= \frac{\Pl(m-1,b+n-u,c)(b+n-u)!}{n!(c-1)!(b+m+n-u-1)!}\\
    & \times (b+2-i)_n (b+n+2-u-i)_{c-1}(i)_{m-1}\\[5pt]
    \M(H_{a,b,c,t,(u+i)_1^m, (v+j)_1^n}) &= \frac{\Pl(a,b+n,c+m)\Pl(u,b+n-u,m)(v+n)!(a+v)!}{\Pl(a+u,b+n-u,m)\Pl(a+v,1,n)(v)!(a+v+n)!} \\
    &\times \sum_{i=1}^{b+n+1-u}\M(R_i)
\end{align*}
\endgroup

A similar calculation {could} be made when $1 \leq \underline{v_n} \leq m$, indexing over the positions of $|\underline{v_n}|$ distinct lozenges which cross the split-line. 

We can apply this method to a different family of dented hexagons, with $v$ arbitrary, $u<b+1$, and $n=1$, employing a split-line which cuts southwest from the eastern dent, as depicted in Figure \ref{fig:fig11}.

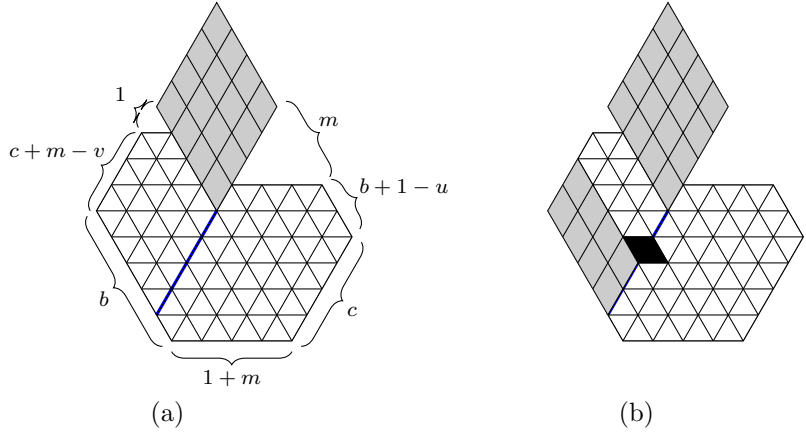
\begin{figure}
\begin{center}

\begin{tikzpicture}[x={(0:1cm)},y={(120:1cm)},z={(240:1cm)},scale=.4]

    \begin{scope}
    
    \draw[very thick, blue](0,-4,4)-- ++(0,0,4);
    
    \draw [decorate,decoration={brace,amplitude=5pt,raise=2pt}]
    (0,0,4.9) -- ++(0,0,-1) node [black,midway,yshift=.3cm,xshift=-.4cm] {\footnotesize
    $1$};
    
    \draw [decorate,decoration={brace,amplitude=5pt,raise=2pt}]
    (0,0,8.1) -- ++(0,0,-3) node [black,midway,yshift=.3cm,xshift=-.8cm] {\footnotesize
    $c+m-v$};
    
    \draw [decorate,decoration={brace,amplitude=5pt,raise=4pt}]
    (0,-4,0) -- ++(0,-2.9,0) node [black,midway,yshift=.3cm,xshift=.4cm] {\footnotesize
    $m$};
    
    \draw [decorate,decoration={brace,amplitude=5pt,raise=4pt}]
    (0,-7,0) -- ++(0,-1.9,0) node [black,midway,yshift=.3cm,xshift=.9cm] {\footnotesize
    $b+1-u$};
    
    \draw [decorate,decoration={brace,amplitude=5pt,raise=4pt}]
    (0,-9,0) -- ++(0,0,4) node [black,midway,yshift=-.3cm,xshift=.4cm] {\footnotesize
    $c$};
    
    \draw [decorate,decoration={brace,amplitude=5pt,raise=4pt}]
    (0,-9,4) -- ++(-4,0,0) node [black,midway,yshift=-.5cm] {\footnotesize
    $1+m$};
    
    \draw [decorate,decoration={brace,amplitude=5pt,raise=4pt}]
    (-4,-9,4) -- ++(0,5,0) node [black,midway,yshift=-.3cm,xshift=-.4cm] {\footnotesize
    $b$};

    \hexTwoDent{0}{5}{4}{4}{3}{4}{1}
    
    \draw[fill=black!20] (0,0,0) -- ++(0,0,4) -- ++(0,-4,0,) -- ++(0,0,-4) -- cycle;
    
    \draw (0,0,1) -- ++(0,-4,0);
    \draw (0,0,2) -- ++(0,-4,0);
    \draw (0,0,3) -- ++(0,-4,0);
    
    \draw (0,-1,0) -- ++(0,0,4);
    \draw (0,-2,0) -- ++(0,0,4);
    \draw (0,-3,0) -- ++(0,0,4);
    \end{scope}
    
    \begin{scope}[xshift=15cm]
    
    \draw[very thick, blue](0,-4,4)-- ++(0,0,4);
    
    \draw[fill=black] (4,0,10) -- ++(1,0,0) -- ++(0,1,0) -- ++ (-1,0,0) -- cycle;

    \hexTwoDent{0}{5}{4}{4}{3}{4}{1}
    
    \draw[fill=black!20] (0,0,0) -- ++(0,0,4) -- ++(0,-4,0,) -- ++(0,0,-4) -- cycle;
    
    \draw (0,0,1) -- ++(0,-4,0);
    \draw (0,0,2) -- ++(0,-4,0);
    \draw (0,0,3) -- ++(0,-4,0);
    
    \draw (0,-1,0) -- ++(0,0,4);
    \draw (0,-2,0) -- ++(0,0,4);
    \draw (0,-3,0) -- ++(0,0,4);

    \draw[fill=black!20] (0,0,6) -- ++(0,0,2) -- ++(0,-4,0,) -- ++(0,0,-2) -- cycle;
    
    \draw (0,-1,6) -- ++(0,0,2);
    \draw (0,-2,6) -- ++(0,0,2);
    \draw (0,-3,6) -- ++(0,0,2);
    
    \draw (0,0,7) -- ++(0,-4,0);    \end{scope}
    
\end{tikzpicture}

\vspace{-.5cm}

\begin{multicols}{2}

(a)

(b)
\end{multicols}
\end{center}

\caption{(a) $H_{0,b,c, m+1,\vec u, \vec v}$ with $n=1$ and a split-line; (b) $R_3$ with forced lozenges shaded grey.
} 
    \label{fig:fig11}
\end{figure}
Again, each tiling of this region has exactly one lozenge which crosses the split-line. If $R_i$ is obtained by removing the lozenge in the $i$th position from the southwest end of the split-line, then tilings of $H_{0,b,c,m+1,(u+i)_1^m, (v+1)}$ are in bijection with the union of the tilings of $\{R_i\}_{i=1}^{c+m-v+1}$:
\begin{align*}
    \M(R_i) &= \M(H_{1,u-1,c+m+1-v-i}) \M(H_{b-u,c,m,1,(i),\vec \emptyset})\\
    &=\frac{\Pl(b-u,c,m+1)c!}{(u-1)!(b-u+c)!m!}\\
    & \times(c+m+2-v-i)_{u-1}(c-i+2)_{m}(i)_{b-u}\\[5pt]
    H_{a,b,c,m+1,(u+i)_1^m, (v+1)} &= \frac{\Pl(a,b+1,c+m)\Pl(u,b+1-u,m)v!(c+m+a)!c!}{\Pl(a+u,b+1-u,m)(c+m)!(a+v)!}\\
    & \times \sum_{i=1}^{c+m-v+1} \M(R_i).
\end{align*}

We note some dead ends. We have shown that the tiling function for the dented hexagon $H_{a,b,c,t,\vec u, \vec v}$ may be given as a polynomial of {entirely linear factors} (of $a$ when the other parameters are fixed). The same is not true for obvious other parameters of the region, such as $c$ or $b$. 
Similarly when dents were placed along more than two sides of the hexagon, or along \emph{short} sides of the hexagon, the tiling function of the region could not be interpreted as a polynomial of linear factors over any obvious single parameter of the region.

\medskip

\textbf{Acknowledgments.} The author would like to thank M. Ciucu for introducing me to this subject, for his insights into this problem, and for his feedback on this paper.

\end{document}